\documentclass[letterpaper, 10 pt, conference]{ieeeconf}  % Comment this line out
\IEEEoverridecommandlockouts                              % This command is only
\usepackage{amsmath,amssymb}
\usepackage{mathtools} %←追加したパッケー

\usepackage{algorithm}
\usepackage{algpseudocode}% http://ctan.org/pkg/algorithmicx

\usepackage{graphicx}
\usepackage{caption}
\usepackage{xcolor}
\usepackage{setspace}
\usepackage{url}
\usepackage{here}
\usepackage{booktabs}
\usepackage{multirow}
\usepackage{multicol}
\usepackage{bigstrut}
\usepackage{cases}
\usepackage{comment}
\usepackage{cite}
\usepackage{framed}
\newtheorem{theorem}{Theorem}
\newtheorem{lemma}{Lemma}
\newtheorem{exam}{Example}
\newtheorem{prop}{Proposition}
\newtheorem{rem}{Remark}

\newtheorem{assu}{Assumption}

\usepackage{mathtools} %←追加したパッケージ
\DeclareMathOperator*{\argmin}{arg\,min}%←追加したパッケージ
\newcommand{\pd}{\partial}

\newcommand{\CC}{\mathcal{C}}
\newcommand{\CN}{\mathcal{N}}
\newcommand{\CD}{\mathcal{D}}
\newcommand{\CE}{\mathcal{E}}

\newcommand{\BRnd}{\mathbb{R}^{nd}}
\newcommand{\BRdn}{\mathbb{R}^{dn}}

\newcommand{\QG}{\bar{\mathcal{Q}}_\mathcal{G}}
\newcommand{\QGactive}{\mathcal{Q}_\mathcal{G}}
\newcommand{\QiGactive}{\mathcal{Q}_\mathcal{G}^{i}}
\newcommand{\QjGactive}{\mathcal{Q}_\mathcal{G}^{j}}
\newcommand{\maxQG}{\mathcal{Q}_\mathcal{G}^\mathrm{max}}

\newcommand{\QiG}{\bar{\mathcal{Q}}^{i}_\mathcal{G}}

\newcommand{\CG}{\mathcal{G}}
\newcommand{\CL}{\mathcal{L}}

\newcommand{\Bx}{\mathbf{x}}
\newcommand{\Bxkp}{\mathbf{x}^{k+1}}
\newcommand{\Bu}{\mathbf{u}}
\newcommand{\Bv}{\mathbf{v}}
\newcommand{\Bukp}{\mathbf{u}^{k+1}}
\newcommand{\By}{\mathbf{y}}
\newcommand{\Bc}{\mathbf{c}}
\newcommand{\Bw}{\mathbf{w}}
\newcommand{\Bykp}{\mathbf{y}^{k+1}}

\newcommand{\BA}{\mathbf{A}}
\newcommand{\BB}{\mathbf{B}}

\newcommand{\BD}{\mathbf{D}}
\newcommand{\BL}{\mathbf{L}}
\newcommand{\BP}{\mathbf{P}}
\newcommand{\BQ}{\mathbf{Q}}
\newcommand{\BW}{\mathbf{W}}
\newcommand{\BR}{\mathbb{R}}

\newcommand{\diag}{\text{diag}}

\newcommand{\BGamma}{{\boldsymbol\Gamma}}
\newcommand{\BTheta}{{\boldsymbol\Theta}}

\newcommand{\BPhi}{{\boldsymbol\Phi}}
\newcommand{\CCl}{{\mathcal{C}_l}}
\newcommand{\prox}{\mathrm{prox}}

\newcommand{\subl}{l}
\newcommand{\qstar}{|\mathcal{Q}_G|}
\newcommand{\bdiag}{\text{block-diag}}

\title{\LARGE \bf %\vspace{-1cm}
Distributed Optimization of Clique-wise Coupled Problems 
}

\author{Yuto Watanabe and Kazunori Sakurama% <-this % stops a space
\thanks{This work was partially supported by the joint project of Kyoto University and Toyota Motor Corporation, titled ``Advanced Mathematical Science for Mobility Society''.}% <-this % stops a space
\thanks{Y. Watanabe and K. Sakurama are with the Department of Systems Science, Graduate School of Informatics, Kyoto University, Yoshida-Honmachi, Sakyo-ku, Kyoto 606--8501, Japan, {\tt\small y-watanabe@sys.i.kyoto-u.ac.jp, sakurama@i.kyoto-u.ac.jp}. 
}%
}

\begin{document}

\setlength{\abovedisplayskip}{4pt} % 上部のマージン
\setlength{\belowdisplayskip}{4pt} % 下部のマージン

\maketitle
\thispagestyle{empty}
\pagestyle{empty}

%%%%%%%%%%%%%%%%%%%%%%%%%%%%%%%%%%%%%%%%%%%%%%%%%%%%%%%%%%%%%%%%%%%%%%%%%%%%%%%%
\begin{abstract}
    This study addresses a distributed optimization with a novel class of coupling of variables, called \textit{clique-wise coupling}.
    A clique is a node set of a complete subgraph of an undirected graph.
    This setup is an extension of pairwise coupled optimization problems (e.g., consensus optimization) and allows us to handle coupling of variables consisting of more than two agents systematically.
    To solve this problem, we propose a clique-based linearized ADMM algorithm, which is proved to be distributed.
    Additionally, we consider objective functions given as a sum of nonsmooth and smooth convex functions and present a more flexible algorithm based on the FLiP-ADMM algorithm.
    Moreover, we provide convergence theorems of these algorithms.
    Notably, all the algorithmic parameters and the derived condition in the theorems depend only on local information, which means that each agent can choose the parameters in a distributed manner.
    % Moreover, for objective functions given as a sum of nonsmooth and smooth convex functions, we present a more flexible algorithm based on the FLiP-ADMM algorithm.
    Finally, we apply the proposed methods to a consensus optimization problem and demonstrate their effectiveness via numerical experiments. 
\end{abstract}

\section{Introduction}

In recent years, distributed optimization has attracted much attention in the control, signal processing, and machine learning communities.
In this field, a large body of studies has been dedicated to \textit{pairwise coupled optimization problems}.
In this type of problem, every coupling of variables comprises two agents' decision variables corresponding to the communication path between the two agents.
% , which is expressed by an edge of the communication graph.
The most representative example of this setup is consensus optimization problems, and many studies have presented sophisticated algorithms, such as \cite{Nedic2009-kd,Shi2015-bm, qu2017harnessing, Li2019-lj}.
Recently, \cite{Latafat2019-mb} and \cite{Li2023_primaldual} have investigated
distributed optimization problems with pairwise linear constraints.
Their applications are not limited to consensus optimization but contain formation control, distributed model predictive control, and so on.
Moreover, in the field of multi-agent control, 
various coordination tasks (e.g., rendezvous and formation) have been formulated in a pairwise coupled form \cite{Oh2015-fx, Bullo2009-ka, Sakurama2021-oy}.

This study addresses a more general form of distributed optimization than the conventional pairwise coupled ones to handle coupling of more than two decision variables.
Consider a multi-agent system with $n$ agents over a communication network, expressed by a time-invariant undirected graph $\CG=(\CN,\CE)$ with $\CN=\{1,\ldots,n\}$ and an edge set $\CE$.
Now, we aim to solve the following optimization problem, called a \textit{clique-wise coupled optimization problem}, in a distributed manner:
\begin{align}\label{prob}
\begin{array}{cl}
\underset{
\substack{x_i\in\BR^d,\,i\in\CN\\ y_\subl\in\BR^{q_\subl},\,l\in\QGactive}
}
{\mbox{minimize}}&
\displaystyle \sum_{i\in\CN}  f_i(x_i) + \sum_{l\in\QGactive} g_\subl (y_\subl) \\
\displaystyle\mbox{subject to}&
\underbrace{A_\subl x_\CCl + B_\subl y_\subl = c_\subl \quad \forall l\in \QGactive \subset\QG  }_{
% \substack{
% \text{Clique-wise coupling w.r.t.} \\ \CCl=\{j_1,\ldots,j_{|\CCl|}\}\subset\CN}
\text{Clique-wise coupling w.r.t. }\CCl\subset\CN}
% \displaystyle\mbox{subject to}
% % &x_i\in \CX_i\subset\BR^d\quad i=1,\ldots,n\\
% % &\displaystyle[x_i^\top,x_j^\top]^\top \in \CX_{ij}\subset\BR^{2d}\quad \forall(i,j)\in\CE \\
% & 
% % \displaystyle\underbrace{
% x_{\CC_l}\in \CD_{l}\quad \forall l\in\QGactive
% % }_{
% % \text{Substantial constraints}
% % }
\end{array}
\end{align}
with $A_\subl \in \BR^{p_\subl \times |\CCl|d}$, $B_\subl \in \BR^{p_\subl \times q_\subl}$, and $c_\subl\in\BR^{p_\subl}$,
where $f_i:\BR^d\to\BR\cup\{+ \infty\},\,i\in \CN$ and $g_\subl:\BR^{q_\subl}\to\BR\cup\{+ \infty\},\,l\in\QGactive$ are proper, closed, and convex functions (possibly nonsmooth).
\begin{figure}[t]
    \centering
    \includegraphics[width=0.9\columnwidth]{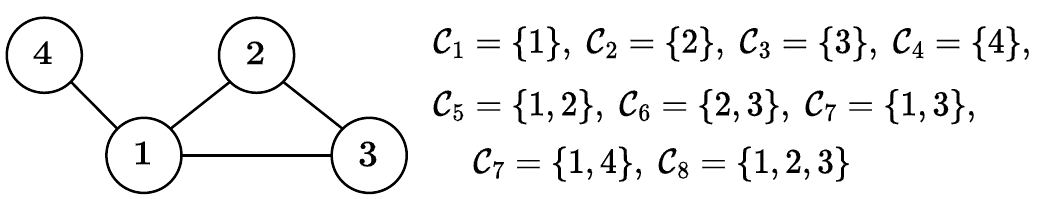}
    \caption{Example of clique.}
    \label{fig:clique_ex}
\end{figure}
The vector $x_i$ represents the decision variable of agent $i$, and $y_\subl$ represents the variable with respect to clique $l$.
For $x_1,\ldots,x_n$, and the set $\CCl=\{j_1,\ldots,j_{|\CCl|}\}\subset\CN$, let $x_{\CCl}$ denote $x_{\CCl}=[x_{j_1}^\top,\ldots,x_{j_{|\CCl|}}^\top]^\top\in\BR^{|\CCl|d}$.
Here, the set $\CCl$ represents a clique, i.e., a complete subgraph in the graph $\CG$ \cite{bollobas1998modern}.
$\QG$ is the index set of all the cliques in $\CG$, and $\QGactive$ is a subset of $\QG$.
For example, in the undirected graph in Fig. \ref{fig:clique_ex}, $\QG=\{1,\ldots,8\}$ holds, and the cliques $\CC_1,\ldots,\CC_8$ are obtained as shown in Fig. \ref{fig:clique_ex}.
Note that the nodes and edges are always cliques, and hence Problem \eqref{prob} always contains conventional pairwise coupled optimization problems.

A remarkable benefit of the clique-wise coupling framework is that we can systematically handle variable couplings of more than two agents.
Possible applications include consensus optimization \cite{Nedic2009-kd, Shi2015-bm, Bertsekas2015-nx, Li2019-lj}, formation control \cite{Oh2015-fx,Latafat2019-mb}, robust PCA (with clique-wise trace norm minimization) \cite{Xu2010-fh,Candes2011-nf}, network lasso (fused lasso, total variation regularization) \cite{Rudin1992-du,Wahlberg2012-uk,Hallac2015-fm}, multi-task learning \cite{Zhang2017-pn}, and (clique-wise) resource allocation \cite{Watanabe2023-xy}, as illustrated in Table \ref{table:applications}.
This table shows concrete formulations of these applications corresponding to Problem \eqref{prob}.
% In robust PCA and multi-task learning, each agent is assumed to have different but co-related data, and the agents cooperatively learn by relaxing low-rank approximation of matrices via the trace norm regularization term, which is coupled beyond cliques in general, into a clique-wise form.

{\renewcommand\arraystretch{1.2}
\begin{table*}[ht]
\footnotesize
 \caption{Practical application examples of Problem \eqref{prob}.}
 \label{table:applications}
\begin{tabular}{|c|c|c|c|c}
\hline
Applications &
  Constraints &
  $f_{[l]}$ &
  \multicolumn{1}{l|}{$f_i$} \\ \hline
Consensus optimization \cite{Nedic2009-kd, Shi2015-bm, qu2017harnessing, Li2019-lj}
&  $x_\CCl-y_\subl=0$
& \begin{tabular}{c}
Indicator function for \\$\CD_l=\{y_\subl: \exists \xi \text{ s.t. } y_\subl =\mathbf{1}_{|\CCl|}\otimes  \xi\}$ \footnotemark[1]
  \end{tabular}
&  \multicolumn{1}{l|}{$f_i(x_i)$} \\ \hline
  
\begin{tabular}{c}
Robust PCA \cite{Xu2010-fh,Candes2011-nf}\\ (clique-wise trace norm minimization) 
 \end{tabular}
&
  $Y_j=S_j+L_j\quad \forall j\in\CCl$ \footnotemark[2]
&
  $\theta_\subl \|L_\CCl\|_*$ &
  \multicolumn{1}{l|}{$\|S_i\|_1$} \\ \hline
\begin{tabular}{c}
Formation control \cite{Bullo2009-ka, Oh2015-fx, Sakurama2021-oy,Latafat2019-mb}
\end{tabular}
&   \begin{tabular}{c}
$x_{i,t+1}=A_i x_{i,t}+B_i u_{i,t}$ \\ $y_{\subl,t}= x_{i,t}-x_{j,t}$ 
\end{tabular}
&
  \begin{tabular}{c}
 $\frac{1}{2} \sum_{t=1}^{T}\| [y_{\subl,t} - r_{ij}\|_{Q_{ij}}^2$ 
\end{tabular}
&  \multicolumn{1}{l|}{$\frac{1}{2} \sum_{t=1}^{T-1} u_{i,t}^\top R_i u_{i,t}$}  \\ \hline
\begin{tabular}{c}
Network lasso  \cite{Rudin1992-du,Wahlberg2012-uk,Hallac2015-fm}
% \\ (Fused Lasso, total variation regularization)
\end{tabular} &
   $x_i-x_j -y_\subl = 0$ &
  $\lambda_{ij} \|y_\subl\|$ &
  \multicolumn{1}{l|}{loss function $\ell_i(x_i)$} \\ \hline
Multi-task learning \cite{Zhang2017-pn}
& 
\begin{tabular}{c}
$x_j - y_{\subl, j}=0$\\ $\forall j\in\CCl=\{ j_1,\ldots,j_{|\CCl|} \}$
\end{tabular}
& $\lambda_\subl \|[y_{\subl, j_1},\ldots, y_{\subl, j_{|\CCl|}} ]\|_*$ 
&
\multicolumn{1}{l|}{loss function $\ell_i(x_i)$} \\ \hline
(Clique-wise) resource allocation \cite{Watanabe2023-xy}
& 
$x_{\CCl} - y_{\subl}=0$
&   \begin{tabular}{c}
Indicator function for \\$\CD_l=\{y_\subl: \mathbf{1}^\top y_\subl = N_\subl>0,\,y_\subl \geq 0 \}$
  \end{tabular}
&
$\frac{1}{2} \|x_i - x_*\|^2$
\\ \hline
\end{tabular}
\end{table*}
}
\normalsize

In this study, we propose a novel distributed algorithm, the Clique-based Linearized ADMM (CL-ADMM) algorithm, for Problem \eqref{prob} based on the framework of the alternating direction method of multipliers (ADMM) \cite{Boyd2011-yu,Bertsekas2015-nx,Ryu2022-oc} with the linearization technique and localized algorithmic parameters.
The proposed method can be implemented with local communication. 
% and its algorithmic parameters are given as an agent-wise or clique-wise form without any global information.
Additionally, we consider that the objective functions $f_i$ and $g_\subl$ are given as a sum of a non-smooth convex and smooth convex functions.
Under this setup, we provide a more flexibly implementable algorithm, called  the Clique-based Linearized FLiP-ADMM (CL-FLiP-ADMM), based on the Function Linearized Proximal ADMM (FLiP-ADMM) algorithm in \cite{Ryu2022-oc}.
Furthermore, through the convergence analysis of the CL-ADMM and CL-FLiP-ADMM, we prove the exact convergence to an optimal solution under fully localized conditions.
Finally, we apply the proposed methods to consensus optimization and demonstrate their effectiveness through numerical experiments.
The experimental result of consensus optimization implies that the clique-wise handling of pairwise constraints can enhance the performance.

The major contributions of this paper are as follows:
    (a) We propose a highly expressive framework in \eqref{prob} suitable for distributed optimization and provide a variety of practical examples as shown in Table \ref{table:applications};
    (b) The algorithmic parameters in the CL-ADMM and CL-FLiP-ADMM are also distributed.
    Moreover, we provide convergence theorems with no global parameter, which means that each agent can choose its parameters in a distributed manner;
    (c) The CL-ADMM and CL-FLiP-ADMM can alleviate the computational burden because the agents in the same clique can share the computation of the proximal mapping of $g_\subl$.

Recently, several studies (e.g., \cite{Chang2016-lz,Falsone2020-mr,Notarnicola2020-bn,Wu2023-vn}) have presented distributed algorithms to solve optimization problems with globally coupled constraints.
This setup allows us to consider constraints involving all decision variables.
However, the existing methods for this setup cannot enjoy our proposed methods' features, such as the contributions (b) and (c) in the above paragraph.
% Moreover, the clique-wise coupled problem in \eqref{prob} can be solved more flexibly than the globally coupled problems because Problem \eqref{prob} can be solved directly by well-known methods (e.g., ADMM \cite{Bertsekas2015-nx,Boyd2011-yu,Ryu2022-oc} and Condat-V\~{u} \cite{Condat2013-ld,Vu2013-ip}) in a distributed manner, differently from the globally coupled problems.
% Note that when we apply these well-known methods to the globally coupled problems, we have to introduce auxiliary variables or to transform the problems into a consensus optimization problem by taking their dual, which can degrade the performance.
Moreover, the clique-wise coupled problem in \eqref{prob} can be solved more flexibly than the globally coupled problems because the primal problem \eqref{prob} can be solved directly by well-known methods (e.g., ADMM \cite{Bertsekas2015-nx,Boyd2011-yu,Ryu2022-oc} and Condat-V\~{u} \cite{Condat2013-ld,Vu2013-ip}) in a distributed manner.
This is advantageous in terms of ease of solving problems.
To apply these well-known methods to the globally coupled problems, we need to introduce additional auxiliary variables (e.g., to take their dual problems to transform them into a consensus optimization).
This approach may increase the size of variables and degrade the performance.
% Note that whether a primal problem or its dual can be solved more efficiently depends on problems \cite{Bertsekas2015-nx}, and 
% Solvability of primal and dual problems often differ, and to apply the well-known methods to the globally coupled problems, we have to 
% can be transformed into a consensus optimization problem by taking their dual formulations.

Finally, notice that Problem \eqref{prob} is more general than optimization problems with clique-wise coupled constraints in the authors' paper \cite{Watanabe2023-xy} because we can handle not only constraints but also regularization terms in Problem \eqref{prob}.
% % This implies that our framework bridges the gap between pairwise coupling and global coupling.
% Additionally, some application examples are newly developed from the perspective of clique, e.g., low-rank approximation of matrices via clique-wise trace norm minimization in the robust PCA and the multi-task learning in Table \ref{table:applications}.

The rest of the paper is organized as follows.
Section \ref{sec:pre} provides preliminaries.
Section \ref{sec:alg} presents the proposed algorithms, and
Section \ref{sec:convergence} presents the convergence theorems.
In Section \ref{sec:case_study}, we apply the proposed methods to consensus optimization.
% Section \ref{sec:proof_convergence} presents the proofs.
Finally, Section \ref{sec:conclusion} concludes our paper.

\footnotetext[1]{
In this formulation, we can include some proximable functions (e.g., $\ell_1$ norm regularization) in $g_\subl$. For details, see Subsection \ref{subsec:consensus}. 
}
\footnotetext[2]{Here, $Y=[Y_1,\ldots,Y_n]$ is a data matrix. For example, $Y$ is a sequence of images, and each column of $Y$ corresponds to each frame.
We assume that $Y$ can be decomposed into a sparse matrix $\hat{S}=[S_1,\ldots,S_n]$ and a low-rank matrix $L=[L_1,\ldots,L_n]$.}
\section{
Preliminaries
}\label{sec:pre}
\subsection{Notation}\label{subsec:notation}
Let $\mathbb{R}$ and $\mathbb{N}$ be the set of real numbers and that of positive integers, respectively.
Let $|\cdot|$ be the number of elements in a countable finite set.
Let $I_d\in \mathbb{R}^{d\times d}$ denote the $d\times d$ identity matrix.
We omit the subscript of $I_d$ when the dimension is obvious.
Let $\mathbf{1}_d = [1,\ldots,1]^\top\in \mathbb{R}^{d}$.
For a $m$-dimensional vector $[a_1,\ldots,a_i,\ldots,a_N]^\top\in\mathbb{R}^N$, $\mathrm{diag}(a)$ denotes the diagonal matrix whose $i$th diagonal entry is $a_i$.
Similarly, for matrices $R_1,\ldots,R_i,\ldots,R_N$, $\bdiag(R_1,\ldots,R_N)$ represents the block diagonal matrix whose $i$th diagonal block is $R_i$. 
For $\mathcal{M}=\{1,\ldots,N\}$, $\bdiag([R_j]_{j\in\mathcal{M}})$ represents $\bdiag(R_1,\ldots,R_N)$.
For $v,\,u\in\BR^m$, and a positive definite and symmetric matrix $Q\in\mathbb{R}^{m\times m}$, $\langle v,u \rangle_Q := v^\top Q u$ denotes the inner product of $v$ and $u$ with respect to $Q$.
Additionally, we define the norm $\|\cdot\|_Q$ as $\| v \|_Q=\sqrt{\langle v,v \rangle_Q }$ for a vector $v\in\mathbb{R}^m$.
When $Q=I_m$, we simply write $\langle \cdot,\cdot \rangle_{I_m}$ and $\|\cdot\|_{I_m}$ as $\langle \cdot,\cdot \rangle$ and $\|\cdot\|$, respectively.
For $v\in\BR^m$, $\|v\|_1$ denotes the $\ell_1$ norm of $v$.
For a matrix $R\in\BR^{d_1\times d_2}$, $\|R\|_*$ denotes the trace norm of $R$, i.e., the sum of its singular values.
For $Q\in\BR^{m \times m}$, let $\lambda_\mathrm{max}(Q)$ and $\lambda_\mathrm{min}(Q)$ be the largest and smallest eigenvalues of $Q$, respectively.
For a vector $\mathbf{v}=[v_1^\top,\ldots,v_j^\top,\ldots,v_{N}^\top]^\top\in\mathbb{R}^{N d}$ with vectors $v_1,\ldots,v_j,\ldots,v_{N}\in\mathbb{R}^d$, $[\Bv]_j$ represents the operation to extract the $j$th vector $v_j$ from $\Bv$, that is,
\begin{align*}
   [\mathbf{v}]_j=v_j\in\mathbb{R}^{d}.
\end{align*}
% For matrices $Q_1,\ldots,Q_N:\BR^{d\times d}$ and $\CC=\{j_1,\ldots,j_{|\CC|}\}\subset\{1,\ldots,n\}$,
% let $[Q_j]_{j\in\CC}\in \mathbb{R}^{Nd\times d}$ be $[Q_j]_{j\in\CC}=[Q_1^\top,\ldots,Q_N^\top]^\top$, where $i_1,\ldots,i_{|\mathcal{I}|} \in \mathcal{I}$ satisfy $1\leq i_1<\cdots<i_{|\mathcal{I}|}\leq n$.
For a vector $\Bx=[x_1^\top,\ldots,x_n^\top]^\top\in\BRnd$ with $x_1,\ldots,x_n\in\mathbb{R}^d$ and  a subset $\CC=\{j_1,\ldots,j_{|\CC|}\}\subset\CN$, let $x_{\CC}$ be 
% \begin{equation*}
    $x_\CC = [x_{j_1}^\top,\ldots,x_{j_{|\CC|}}^\top]^\top \in \mathbb{R}^{|\CC|d}$,
% \end{equation*}
where $\{j_1,\ldots,j_{|\CC|}\}$ is a strictly monotonically increasing sequence.
% For $\CD\subset\BRnd$ and $\CC=\{j_1,\ldots,j_{|\CC|}\}\subset\{1,\ldots,n\}$, the projection of $\CD$ onto $\CC$ is defined as
% \begin{align}\label{eq:proj}
%     \mathrm{proj}_{\CC} (\CD) =& \{y\in \mathbb{R}^{|\CC|d}: \exists x\in \CD\;\mathrm{s.t.}\;y=x_\CC\}.
% \end{align}
% We can rewrite \eqref{eq:proj} as $\mathrm{proj}_{\CC} (\CD)= \{(W_{\CC}\otimes I_d) x \in \mathbb{R}^{|\CC|d}: x\in \CD\}$, where $W_\CC$ is a $|\CC|\times n$ matrix given as follows:
% \begin{equation}\label{W_C}
%     W_{\CC}= [w_{kl}]\in \mathbb{R}^{|\CC|\times n},
%     \quad w_{kl}=\begin{cases}
%         1, & (k,l)=(k,j_k) \\
%         0, & \mathrm{otherwise}.
%     \end{cases}
% \end{equation}
% Thus, $x_\CC=W_\CC x$ holds for $\CC\subset\{1,\ldots,n\}$ and $x\in\BRnd$.
For a differentiable function $f:\BRnd\to\mathbb{R}$ and $\Bx\in\BRnd$, we write $\nabla f(\Bx)=\pd f/\pd \Bx(\Bx)$.
% [\nabla_1 f(\Bx)^\top,\ldots,\nabla_n f(\Bx)^\top]^\top \in \BRnd$ with $\nabla = \pd /\pd \Bx$ and $\nabla_i=\pd/\pd x_i.$

For a proper, closed, and convex function $g:\BR^d:\to\BR\cup\{+ \infty\}$, a positive definite and symmetric matrix $Q\in\BR^{d\times d}$, and $x\in\BR^d$, the proximal mapping of $g$ with respect to $Q$ is represented by $\prox^Q_g(x) = \argmin_{y\in\BR^d} \{ g(y) + (1/2) \|x-y\|_Q^2\}$.
When $Q=I_d$, we write $\prox_{g}^{I_d}(\cdot)=\prox_g(\cdot)$.
When the proximal mapping of $g$ can be computed efficiently, the function $g$ is said to be \textit{proximable}.
% Additionally, consider a convex and differentiable function $h:\BR^d\to\BR$.

\subsection{Graph Theory}
Here, we provide graph-theoretic concepts.
Consider a graph $\CG=(\CN,\CE)$ with a node set $\CN=\{1,\ldots,n\}$ and an edge set $\CE$ consisting of pairs $\{i,j\}$ of different nodes $i,j\in \CN$.
Note that throughout this paper, we consider undirected graphs and do not distinguish $\{i,j\}$ and $\{j,i\}$ for each $\{i,j\}\in\CE$.
% If $(i,j)\in\CE \Leftrightarrow (j,i)\in\CE$ holds for all $(i,j)\in\CE$, the graph $G$ is said to be \textit{undirected}.
% In the following, we consider a time-invariant undirected graph $G$.
% If a sequence $(i_0,i_1,\ldots,i_l)$ of nodes satisfies $l\geq 1$ and $(i_k,i_{k+1})\in\CE,\,k=0,1,\ldots,l-1$, the squence is said to be a \textit{path}.
% A graph $G$ is \textit{connected} when a path exists between any two nodes $(i,j)$ in $\CG$.
For $i\in\CN$ and $G$, let $\CN_i\subset \CN$ be the \textit{neighbor set} of node $i$ over $\CG$, defined as $\CN_i=\{j\in \mathcal{N}:\{i,j\}\in \mathcal{E}\}\cup\{i\}$.
% \begin{equation*}
%     \CN_i=\{j\in \mathcal{N}:(i,j)\in \mathcal{E}\}.
% \end{equation*}

For an undirected graph $\CG$,
consider a set $\CC\subset\CN$.
For $\CC$ and $\CE$,  let $\CE|_{\CC}$ be $\mathcal{E}|_{\mathcal{C}}=\{\{i,j\}\in \mathcal{E}:i,j\in \mathcal{C}\}$. We call $\CG|_\CC=(\CC,\CE|_\CC)$ a subgraph induced by $\CC$.
If $\CG|_\CC$ is complete, $\CC$ is called a \textit{clique} in $\CG$.
We define $\QG=\{1,2,\ldots,q\}$ as the set of indices of all the cliques in $\CG$.
For $\QG$, the set $\QGactive$ represents a subset of $\QG$.
If a clique $\CC$ is not contained by any other cliques, $\CC$ is said to be \textit{maximal}.
Let $\maxQG(\subset\QG)$ be the set of indices of all the maximal cliques in $\CG$. 
For $i\in\CN$, we define $\QiG$ as an index set of all the cliques containing $i$, that is, $\QiG=\{l\in \QG:i\in \mathcal{C}_l \}$.
% \begin{equation*}
%     \mathrm{clq}_i(G)=\{k\in \QG:i\in \mathcal{C}_k \}.
% \end{equation*}
For each $i\in\CN$, $\CN_i$, and $\CC_l,\,l\in\QiG$, 
% the following relationship holds:
\begin{equation}\label{eq:neighbor_clique}
    \CN_i=\bigcup_{l\in \QiG} \CC_l.
\end{equation}
holds \cite{Sakurama2021-oy}.
Note that agent $i$ can independently compute
the cliques that it belongs to, i.e.,
$\CCl,\,l\in\QiG$, from the undirected subgraph $(\CN_i,\CE_i)$ with $\CE_i = \{\{i,j\}\in\CE:j\in\CN_i\}$.
%%%%%%%%%%%%%%%%%%%%%%%%%%%%%%%%%%%%%%%%%%%%%%%%%%%%%%%%%%%%%%%%%%%%%%%%%%%%%%%%%%%%%%%%%%%%%%%%%%%%%%%%%%%%%%%%%%%%%%%%%%%%%%%%%%%%%%%%%%%%%%%%%%%%%%%%%%%%%%%%%%%%%%%%%%%%%%%%%%%%%%%%%%%%%%%%%%%%%%%%%%%%%%%%%%%%%%%%%%

\section{Algorithm Description}\label{sec:alg}
% In this section, we present the proposed methods for Problem \eqref{prob} and its convergence analysis.
\subsection{Clique-based Linearized ADMM Algorithm}\label{subsec:alg}
The proposed algorithm to solve \eqref{prob}, the Clique-based Linearized ADMM (CL-ADMM) algorithm, is illustrated in Algorithm \ref{alg:a-w-Proposed}.
This algorithm is based on the linearized ADMM algorithm \cite{Boyd2011-yu,Ryu2022-oc} and can be implemented in a distributed manner from \eqref{eq:neighbor_clique}.
Here, let $\QiGactive:=\QiG\cap\QGactive$.
Besides, $\alpha_i,\,i\in\CN$ is an agent-wise algorithmic parameter, and $\beta_\subl,\,\gamma_\subl$, and $\phi_\subl>0,\,l\in\QGactive$ are clique-wise parameters.
% and $P_i,\,Q_i\in\BR^{d\times d},\,i\in\CN$ are agent-wise positive-semidefinite matrices.
% The matrix $Q_\subl$ is the block-diagonal matrix defined as $Q_\subl:= \bdiag([Q_j]_{j\in\CCl})\in\BR$.
In Algorithm \eqref{alg:a-w-Proposed}, $y_\subl^k$ represents the estimate of an optimal $y_\subl$, and $u_\subl^k$ represents the dual variable for $l\in\QGactive$
% corresponding to the constraint $ A_\subl x_\CCl+B_\subl y_\subl=c_\subl$
at the $k$th iteration.
In the $x_i$-update in \eqref{alg:xi-update}, $\pi_\subl:\CCl\to\{1,\ldots,|\CCl|\}$ is the one-to-one mapping satisfying 
% \begin{equation}\label{def:pi}
    $\pi_\subl (i_j) = j$
% \end{equation} 
for $\CCl=\{i_1,\ldots,i_j,\ldots,i_{|\CCl|}\}$ with $1\leq i_1<i_j<i_{|\CCl|}\leq n$.
Note that if the $y_\subl$-subproblem in \eqref{alg:yl-update} has multiple solutions, the agents in $\CCl$ must choose the same value as $y_\subl^{k+1}$.
\begin{exam}
Consider $\CCl=\{2,3,5\}$ and $y_\subl= [ y_{\subl,1}^\top, y_{\subl,2}^\top, y_{\subl,3}^\top]^\top\in\BR^{|\CCl|d}$.
Then, $\pi_\subl(2)=1$, $\pi_\subl(3)=2$, and $\pi_\subl(5)=3$ hold.
Besides, for $y_\subl$, $[y_\subl ]_{\pi_\subl(2)}=y_{\subl,1}$, $[y_\subl ]_{\pi_\subl(3)}=y_{\subl,2}$, and $[y_\subl ]_{\pi_\subl(5)}=y_{\subl,3}$ are obtained.
\end{exam}
\begin{algorithm}[t]
\caption{
% Clique-based projected gradient descent (CPGD)
Clique-based Linearized ADMM (CL-ADMM) Algorithm
% Procedure for agent $i$ with the proposed method
}
\label{alg:a-w-Proposed}
\begin{algorithmic}[1]
% \small 
\Require $x_i^0$, $\alpha_i>0$, $y_\subl^0$, $u_\subl^0$, $\beta_\subl>0$, $\gamma_\subl>0$, and $\phi_\subl>0$ for all $l\in\QiGactive\:(:=\QiG\cap\QGactive)$.
\For {$k = 0,1,\ldots$}
% \small
\State Update $x_i^k$ by 
\begin{subequations}\label{alg:Proposed_update-rule}
\begin{align}\label{alg:xi-update}
% \small
x_i^{k+1} &=\textstyle \prox_{\alpha_i f_i}  
\bigl( x_i^k-\alpha_i \sum_{l\in\QiGactive} 
\bigl[ A_\subl^\top  u_\subl^k \nonumber\\
&\quad\quad +\gamma_\subl A_\subl^\top  \left( A_\subl x_\CCl^k  + B_\subl y_\subl^k   - c_\subl \right) \bigr]_{\pi_\subl(i)} \bigr).
\end{align}
% with $\pi_\subl(\cdot)$ in \eqref{def:pi}.
\State Send $x_i^{k+1}$ to all agents in $\CN_i\setminus \{i\}$.
\State Update $y_\subl^k$ and $u_\subl^k$ for all $ l\in\QiGactive$ by
\begin{align}\label{alg:yl-update}
% \small
y_\subl^{k+1} &= \prox_{\beta_\subl g_\subl} %^{(\diag(\beta_\CCl)\otimes I_d)^{-1}} 
\bigl(y_\subl^k 
- \beta_\subl \bigl(B_\subl^\top u_\subl^k \nonumber\\
&\quad\quad\quad\quad\quad + \gamma_\subl
B_\subl^\top \bigl( A_\subl x_\CCl^{k+1} + B_\subl y_\subl^k - c_\subl 
\bigr) \bigr) \bigr)\\
%\quad \forall l\in\QiGactive
% \end{align}
% \begin{equation}
\label{alg:ul-update}
% \small
u_\subl^{k+1}&=u_\subl^k+ \phi_\subl \gamma_\subl
    \left(A_\subl x_\CCl^{k+1} + B_\subl y_\subl^{k+1} -c_\subl \right).
\end{align}
\end{subequations}
\EndFor\normalsize
\end{algorithmic}
\end{algorithm}

% \subsection{Aggregated form of Problem \eqref{prob}}\label{subsec:problem}

We provide an interpretation of Algorithm \ref{alg:a-w-Proposed}.
First, we give the aggregated form of Problem \eqref{prob} as follows:
% , and let $|\QGactive|=\qstar$
\begin{align}\label{prob:aggregated}
\begin{array}{ll}
\underset{
\Bx,\,\By
}{\mbox{minimize}}
&
\displaystyle F(\Bx) + G(\By)\\
\displaystyle\mbox{subject to}
&\BA \BW \Bx + \BB \By = \Bc,
% \begin{bmatrix}
%    y^{[1]} \\
%    \vdots \\
%    y^{[\qstar]}
% \end{bmatrix}
\end{array}
\end{align}
where 
% \begin{align}
% \label{def:FaFc}
    $F(\Bx) = \sum_{i\in\CN} f_i(x_i)$,
    $G(\By) = \sum_{l\in\QGactive} g_\subl (y_\subl)$, 
% \end{align}
$\Bx=[x_1^\top,\ldots,x_n^\top]^\top$, $\By=[y_1^\top,\ldots,y_{\qstar}^\top]^\top$, $\BA=\bdiag([A_\subl]_{l\in\QGactive})$, $\BB=\bdiag([B_\subl]_{l\in\QGactive})$, and $\Bc=[c_1^\top,\ldots,c_{\qstar}^\top]^\top$.
Now, $\BW$ is defined as the matrix satisfying the following relationships for any $\Bx\in\BRdn$:
\begin{align}\label{def:W}
    &\BW    = \left[ W_1^\top,\ldots,W_{\qstar}^\top \right]^\top
%     \begin{bmatrix}
%    W^{[1]} \\
%    \vdots \\
%    W^{[\qstar]}
% \end{bmatrix} 
\in \BR^{\left(\sum_{l\in\QGactive}|\CCl|\right)d\times nd},\nonumber\\
&W_\subl \Bx = x_\CCl \quad \forall l\in \QGactive.
\end{align}
Then, the Lagrangian function is obtained as follows:
\begin{equation}\label{lagrangian}
\CL(\Bx,\By,\Bu) = F(\Bx)+G(\By)
+ \langle \Bu, \BA\BW \Bx+\BB\By-\Bc \rangle.
\end{equation}
% In this paper, we impose the following assumption.
% % \subsection{Linearized ADMM algorithm}

% \begin{assu}\label{assumption:initialization}
% % \begin{itemize}
%     % \item[b)] 
%     The Lagrangian $\CL$ in \eqref{lagrangian}
% has a saddle point.
% % \end{itemize}
% \end{assu}

By applying the linearized ADMM algorithm \cite{Boyd2011-yu,Ryu2022-oc} to the problem in \eqref{prob:aggregated}, we obtain the following algorithm, where the scalar-valued constants in the normal linearized ADMM are replaced by the diagonal matrices $\BD_\alpha$, $\BD_\beta$, $\BGamma$, and $\BPhi$:
% By Proposition \ref{lem:WTu}, Algorithm \ref{alg:a-w-Proposed} can be rewritten as follows:
\begin{subequations}\label{alg:Full-Proposed}
\begin{align}
\label{alg:Bx_update}
 \Bx^{k+1} &= \prox_{F}^{\BD_\alpha^{-1}}( \Bx^k - \BD_\alpha \BW^\top \BA^\top (\Bu^k \nonumber \\
 & \quad\quad\quad\quad\quad \quad
 + \BGamma (\BA\BW\Bx^k+\BB\By^k-\Bc) )) \\
 \By^{k+1} &= \prox_{G}^{\BD_\beta^{-1}} ( \By^k - \BD_\beta \BB^\top (\Bu^k \nonumber 
 \\ &\quad\quad\quad\quad\quad \quad
 +\BGamma(\BA\BW\Bx^{k+1}+\BB\By^k-\Bc))  ) \\
 \Bu^{k+1} &= \Bu^k + \BPhi \BGamma (\BA\BW\Bx^{k+1}+\BB\By^{k+1}-\Bc).
%  \underset{\Bx}{\operatorname{argmin}}
%  \biggl\{F (\Bx)+\left\langle\nabla F_1\left(\Bx^k\right)+\BW^{\top} \Bu^k, \Bx\right\rangle \nonumber  
%  &+\frac{1}{2}\left\|\BW\Bx-\By^k\right\|^2_\BGamma+\frac{1}{2}\left\|\Bx-\Bx^k\right\|_\BP^2 \biggr\} \\
%   \label{alg:By_update}
% \By^{k+1} \in &\underset{\By}{\operatorname{argmin}}\biggl\{G^1(\By)+\left\langle\nabla G^2\left(\By^k\right)-\Bu^k, \By\right\rangle \nonumber \\
% &\!\!+\frac{1}{2}\left\|\BW\Bx^{k+1}-\By\right\|^2_\BGamma+\frac{1}{2}\left\|\By-\By^k\right\|_\BQ^2\biggr\}  \\
% \label{alg:Bu_update}
%  \Bu^{k+1}=&\Bu^k+\BPhi \BGamma \left(\BW\Bx^{k+1}-\By^{k+1}\right),
\end{align}
\end{subequations}
Then, by setting 
\begin{align*}
    &\BD_\alpha = \bdiag (\alpha_1 I_{d},\ldots,\alpha_n I_{d}) \\
    &\BD_\beta = \bdiag ([ \beta_\subl I_{q_\subl}]_{j\in\QGactive}) \\
    &\BGamma = \bdiag ( [ \gamma_\subl I_{p_\subl}]_{j\in\QGactive}) \\
    &\BPhi = \bdiag ( [ \phi_\subl I_{p_\subl}]_{j\in\QGactive})
\end{align*}
with $\alpha_i>0,\,i\in\CN$ and $\beta_\subl,\,\gamma_\subl,\,\phi_\subl>0,\,l\in\QGactive$,
% $\BD_\alpha = \bdiag (\alpha_1 I_{d},\ldots,\alpha_n I_{d})$ for $\alpha_i>0,\,i\in\CN$, $\BD_\beta = \bdiag ([ \beta_\subl I_{q_\subl}]_{j\in\QGactive})$ for $\beta_\subl>0,\,l\in\QGactive$, $\BGamma = \bdiag ( [ \gamma_\subl I_{p_\subl}]_{j\in\QGactive})$ for $\gamma_\subl>0,\,l\in\QGactive$, and $\BPhi = \bdiag ( [ \phi_\subl I_{p_\subl}]_{j\in\QGactive})$ for $\phi_\subl>0,\,l\in\QGactive$,
we obtain the CL-ADMM in Algorithm \ref{alg:a-w-Proposed} from Algorithm \eqref{alg:Full-Proposed}.
This follows from the following proposition.
% This feature is a key to the distributedness of Algorithm \ref{alg:a-w-Proposed}.
% The proof is presented in Section \ref{sec:proof_Wu}
\begin{prop}\label{lem:WTu}
For all $i\in\CN$, $\CCl,\,l\in\QiGactive$, $\BW$ in \eqref{def:W}, the mapping $\pi_\subl:\CCl\to\{1,\ldots,|\CCl|\}$, and $\Bv=[v_1^\top,\ldots,v_\subl^\top,\ldots,v_{\qstar}^\top]^\top$ with any $v_\subl\in\BR^{|\CCl|d}$, the following equality is satisfied:
\begin{equation*}
    [\BW^\top \Bv]_i = \sum_{l\in\QiGactive} [v_\subl]_{\pi_\subl(i)} \in \BR^d.
\end{equation*}
\begin{proof}
       From \eqref{def:W}, the $(j,i)$ block of $W_\subl \in\BR^{|\CCl|d\times nd}$ can be written as $[W_\subl]_{ji}= w_{\subl,ji}I_d $, where
\begin{equation}\label{def:w_ij}
% \small
    w_{\subl,ji} = 
    \begin{cases}
    1,& i\in\CCl \text{ and }\pi_\subl(i)=j\\
    0,& \text{otherwise}
    \end{cases}.
\end{equation}
% for $\CCl=\{m_{l,1},\ldots,m_{l,|\CCl|}\},\,1\leq m_{l,1}<\ldots<m_{l,|\CCl|}\leq n$.
Additionally, we have $\BW^\top\Bu = \sum_{l\in\QGactive} (W_\subl)^\top u_\subl$.
Then, for the $i$th block $[\BW^\top\Bv]_i\in\BR^d$ of $\BW^\top\Bv$, we obtain $[\BW^\top\Bv]_i 
    = \sum_{l\in\QGactive}\left( \sum_{j=1}^{|\CCl|} [W_\subl]_{ji} [v_\subl]_j \right) 
    = \sum_{l\in\QiGactive} [v_\subl]_{\pi_\subl(i)}$
% \begin{align*}
% \small
% \textstyle
%      [\BW^\top\Bv]_i 
%     = \sum_{l\in\QGactive}\left( \sum_{j=1}^{|\CCl|} [W_\subl]_{ji} [v_\subl]_j \right) 
%     = \sum_{l\in\QiGactive} [v_\subl]_{\pi_\subl(i)}
% \end{align*}
because $[W^\subl]_{ji} [v^\subl]_{j}= [v^\subl]_{\pi_\subl(i)}$ holds for $j$ satisfying $\pi_\subl(i)=j$, and $[W^\subl]_{ji} [v^\subl]_{j}= 0$ holds otherwise.
\end{proof}
% $[\BW^\top \Bv]_i = \sum_{l\in\QiGactive} [v_\subl]_{\pi_\subl(i)} \in \BR^d$ 
% the following relationship holds:
% \begin{equation}\label{ui}
% [\BW^\top \Bv]_i = \sum_{l\in\QiGactive} [v_\subl]_{\pi_\subl(i)} \in \BR^d.
% % \end{bmatrix}
% \end{equation}
% , where $\pi_\subl:\CCl\to\{1,\ldots,|\CCl|\}$ is the one-to-one mapping satisfying $\pi_\subl (i_j) = j$ for $\CCl=\{i_1,\ldots,i_j,\ldots,i_{|\CCl|}\}\,(\subset\CN)$ with $1\leq i_1<i_j<i_{|\CCl|}\leq n$.
\end{prop}

% We have two comments on Algorithm \ref{alg:a-w-Proposed} and \eqref{alg:Full-Proposed}.
% First, using the linearization technique is essential for distributed implementation of the proposed method because this technique allows us to eliminate a global coupling of decision variables that appears in the $x_i$-subproblem in \eqref{alg:xi-update}.

\begin{rem}\label{rem:share}
One of the advantages of the CL-ADMM in Algorithm \ref{alg:a-w-Proposed} is that
 the agents in the same clique $\CCl$ can share the computation of the $y_\subl$-subproblem in \eqref{alg:yl-update}.
Hence, we can alleviate computational burdens per iteration by allocating the computation.
% , for example, through the following assignment problem:
% \begin{align*}
% % \setlength{\baselineskip}{30pt}
% \begin{array}{ll}
% \underset{
% \substack{z_{il},\,i\in\CN,\\ l\in\QGactive
% }
% }{\mbox{minimize}}&
%  \sum_{i=1}^{n} \sum_{l\in\QGactive} z_{ij}
% \\
% \displaystyle\mbox{subject to}
% & \sum_{i\in\CCl} z_{il} =1 \quad l\in\QGactive \vspace{0.1cm}\\
% & \sum_{l\in\QiGactive} c_l z_{il} \leq p_i \quad  i\in\CN\\
% & z_{il}\in\{0,1\}\quad  i\in\CN,\,l\in\QGactive,
% \end{array}
% \end{align*}
% where $p_i>0$ is the maximal computational cost that agent $i$ can bear, and
% $c_l>0$ represents the computational cost of the $y^{[l]}$-subproblem \eqref{alg:yl-update}.
% This feature differs from that introduces 
\end{rem}

\begin{rem}
    The proposed algorithm is based on the linearized ADMM, which is essential for its distributed implementation.
    This is because the augmented Lagrangian of \eqref{lagrangian} can be separated into an agent-wise form by eliminating its coupled terms by the linearization technique.
\end{rem}

\subsection{Clique-based Linearized FLiP-ADMM Algorithm}

We provide a more flexible algorithm, the Clique-based Linearized FLiP-ADMM (CL-FLiP-ADMM) algorithm, than Algorithm \ref{alg:a-w-Proposed} based on the FLiP-ADMM algorithm in \cite{Ryu2022-oc}.
Suppose that the objective functions $f_i$ and $g_\subl$ can be separated as
\begin{equation}\label{def:objective_smooth}
    f_i = f_{i}^1 + f_{i}^2,\quad g_\subl = g_\subl^1 + g_\subl^2,
    % f_i = f_{i}^1 + \underbrace{f_{i}^2}_{\text{smooth}},\quad g_\subl = g_\subl^1 + \underbrace{g_\subl^2}_{\text{smooth}},
\end{equation}
where $f_i^1$ and $g_\subl^1$ are proper, closed, and convex functions, and $f_i^2$ and $g_\subl^2$ are proper, closed, convex, and smooth functions.

Now, we present the Clique-based Linearized FLiP-ADMM (CL-FLiP-ADMM) algorithm in Algorithm \ref{alg:FLiP_linearized}.
This algorithm can be implemented distributedly in the same manner as Algorithm \ref{alg:a-w-Proposed}.
% from Proposition \ref{lem:WTu}.
When $f_i^2=g_\subl^2=0$, this algorithm is reduced to Algorithm \ref{alg:a-w-Proposed}.
Notably, this algorithm allows us to avoid computing the proximal map involving $f_i^2$ and $g_\subl^2$.
% by using the FLiP-ADMM algorithm with the linearization technique.

Note that although the FLiP-ADMM algorithm with the linearization technique is named as doubly linearized ADMM in \cite{Ryu2022-oc}, we refer to Algorithm \ref{alg:FLiP_linearized} as CL-FLiP-ADMM for the sake of consistency with Algorithm \ref{alg:a-w-Proposed}.
% Note that when $f_i^2=g_\subl^2=0$, this algorithm is equivalent to Algorithm \ref{alg:a-w-Proposed}.
% instead of the proposed method in Algorithm \ref{alg:a-w-Proposed} and \eqref{alg:Full-Proposed}.

\begin{algorithm}[t]
\caption{
% Clique-based projected gradient descent (CPGD)
Clique-based Linearized FLiP-ADMM (CL-FLiP-ADMM) Algorithm
% Procedure for agent $i$ with the proposed method
}
\label{alg:FLiP_linearized}
\begin{algorithmic}[1]
% \small 
\Require $x_i^0$, $\alpha_i>0$, $y_\subl^0$, $u_\subl^0$, $\beta_\subl>0$, $\gamma_\subl>0$, and $\phi_\subl>0$ for all $l\in\QiGactive\:(:=\QiG\cap\QGactive)$.
\For {$k = 0,1,\ldots$}
% \small
\State Update $x_i^k$ by 
\begin{subequations}
\begin{align}\label{alg:xi_update_FLiP}
% \small
x_i^{k+1} &= \prox_{\alpha_i f_i^1}  
\bigl( x_i^k-\alpha_i  \bigl( \nabla f_i^2(x_i^k) + \sum_{l\in\QiGactive}\bigl[ A_\subl^\top  u_\subl^k 
\nonumber \\
&\quad\quad
+  \gamma_\subl A_\subl^\top  \bigl( A_\subl x_\CCl^k  + B_\subl y_\subl^k   - c_\subl \bigr) \bigr]_{\pi_\subl(i)} \bigr)
\bigr).
\end{align}
% with $\pi_\subl(\cdot)$ in \eqref{def:pi}.
\State Send $x_i^{k+1}$ to all agents in $\CN_i\setminus \{i\}$.
\State Update $y_\subl^k$ and $u_\subl^k$ for all $ l\in\QiGactive$ by
\begin{align}\label{alg:yl-update_FLiP}
% \small
y_\subl^{k+1} &= \prox_{\beta_\subl g_\subl^1} %^{(\diag(\beta_\CCl)\otimes I_d)^{-1}} 
\bigl(y_\subl^k 
- \beta_\subl  \bigl( \nabla g_\subl^2(y_\subl^k)+ B_\subl^\top u_\subl^k \nonumber\\
&\quad\quad + \gamma_\subl B_\subl^\top
\bigl(A_\subl x_\CCl^{k+1} + B_\subl y_\subl^k - c_\subl
\bigr) \bigr) \bigr)\\
% \end{align}
% \State Update $u_\subl^k$ for all $l\in\QiGactive$ by
% \begin{equation}
\label{alg:ul-update_FLiP}
% \small
    u_\subl^{k+1}&=u_\subl^k+ \phi_\subl \gamma_\subl
    \left(A_\subl x_\CCl^{k+1} + B_\subl y_\subl^{k+1} -c_\subl \right).
\end{align}
\end{subequations}
\EndFor\normalsize
\end{algorithmic}
\end{algorithm}

\section{Convergence Analysis}\label{sec:convergence}

This section presents the key convergence theorems of the CL-ADMM in Algorithm \ref{alg:a-w-Proposed} and CL-FLiP-ADMM in Algorithm \ref{alg:FLiP_linearized}.
% , tailored to distributed implementation.
We now assume the following assumption.
\begin{assu}\label{assu:convergence}
The following statements hold:
\begin{itemize}
% \item[(a)] \textcolor{blue}{The functions $F$ and $G$ are closed convex and proper.}
    \item[(a)] $f_i:\BR^d\to\BR\cup\{+ \infty\},\,i\in\CN$ is proper, closed, and convex. Additionally, for $f_i$ of the form in \eqref{def:objective_smooth}, $f_i^2:\BR^d\to\BR$ is convex and differentiable, and the gradient of $f_i^2$ is $L_{f_i^2}$-Lipschitz continuous, i.e., $ \|\nabla f_i^2(x_i)-\nabla f_i^2(z_i)\| \leq L_{f_i^2}\|x_i-z_i\|$ holds for any $x_i,\,z_i\in\BR^{d}$ and some $L_{f_i^2}\geq0$.
    \item[(b)] $g_\subl:\BR^{q_\subl}\to\BR\cup\{+ \infty\},\,l\in\QGactive$ is proper, closed, and convex. Additionally, for $g_\subl$ of the form in \eqref{def:objective_smooth}, $g_\subl^2:\BR^{q_\subl}\to\BR$ is convex and differentiable, and the gradient of $g_\subl^2$ is $L_{g_\subl^2}$-Lipschitz continuous with some $L_{g_\subl^2}\geq0$.
    % that is, $ \|\nabla g_\subl^2(y_\subl)-\nabla g_\subl^2(z_\subl)\| \leq L_{g_\subl^2}\|y_\subl-z_\subl\|$ for any $y_\subl,\,z_\subl\in\BR^{q_\subl}$ and
    % some $L_{f_\subl^2}\geq0$.
    \item[(c)] The Lagrangian $\CL$ in \eqref{lagrangian} has a saddle point.
    % \item[(d)]
    % In Algorithm \ref{alg:a-w-Proposed} and \ref{alg:FLiP_linearized}, the $x_i$- and $y_\subl$- subproblems for all $i\in\CN$ and $l\in\QGactive$ always have solutions.
\end{itemize}
\end{assu}

First, we provide the following theorem.
In this theorem, 
% all the parameters and the conditions are given as an agent-wise or clique-wise form, which means that 
each agent can check the conditions in a distributed fashion.
Regarding the choice of the parameters, a detailed discussion is available in Chapter 8 of \cite{Ryu2022-oc}.
% Note that Algorithm \ref{alg:a-w-Proposed} corresponds to the case of $f_i^2=g_\subl^2=0$, i.e., $L_{f_i^2}=L_{g_\subl^2}=0$ in the FLiP-ADMM algorithm in \ref{alg:FLiP_linearized}.
\begin{theorem}\label{thm:convergence}
Consider Algorithms \ref{alg:a-w-Proposed} and \ref{alg:FLiP_linearized}.
Assume that Assumption \ref{assu:convergence} is satisfied.
% Assume that all the subproblems of $x_i,\,i\in\CN$ and $y_\subl,\,l\in\QGactive$ always have solutions.
% Assume that for all $i\in\CN$, $f_i$ is $L_{f_i}$-smooth, and for all $l\in\QGactive$, $g_\subl$ is $L_{g_\subl}$-smooth, where $L_{f_i}\geq0$ and $L_{g_\subl}\geq0$.
Assume that for all $i\in\CN$, all $l\in\QGactive$, and some $\varepsilon_\subl \in(0,2-\phi_\subl)$, the following inequalities hold:
\begin{subequations}\label{stepsize_condition}
\begin{align}
\label{stepsize_condition1}
    &\alpha_i^{-1} \geq \sum_{l\in\QiGactive}\gamma_\subl \lambda_\mathrm{max}(A_\subl^\top A_\subl) + L_{f^2_i} ,\quad \\
\label{stepsize_condition2}
    &\beta_\subl^{-1}-  \gamma_\subl \lambda_\mathrm{max}(B_\subl^\top B_\subl) \geq 0
    , \quad  \\
\label{stepsize_condition3}
    &\gamma_\subl\left(1-\frac{(1-\phi_\subl)^2}{2-\phi_\subl -\varepsilon_\subl }\right)B_\subl^\top B_\subl +Q_\subl \succeq 3 L_{g^2_\subl} I_{q_\subl},
\end{align}
\end{subequations}
where $Q_\subl = \beta_\subl^{-1}I_{q_\subl}- \gamma_\subl B_\subl^\top B_\subl$.
Then, $\lim_{k\to\infty}(F(\Bx^k)+G(\By^k))=(F(\Bx^*)+G(\By^*))$ and $\lim_{k\to\infty}(\BA\BW\Bx^k+\BB\By^k-\Bc)=0$ hold, where $(\Bx^*,\By^*,\Bu^*)$ is a saddle point of the Lagrangian function $\CL$.
\end{theorem}
\begin{proof}
%     % See Section \ref{sec:proof_convergence}.
% We prove Theorem \ref{thm:convergence} based on the FLiP-ADMM framework in \cite{Ryu2022-oc} with distributed algorithmic parameters.
We prove Theorem \ref{thm:convergence} based on the idea of the proof of Theorem 6 in \cite{Ryu2022-oc} and the lemmas in Appendix \ref{subsec:supporting_lemma}.
% Let $\BP = \bdiag(P_1,\ldots,P_n)$ with some positive semi-definite matrix $P_i\in\BR^{d\times d},\,i\in\CN$, and let $\BQ= \bdiag([Q_\subl]_{l\in\QGactive})$ with some positive semi-definite matrix $Q_\subl \in \BR^{q_\subl\times q_\subl},\,l\in\QGactive$.
% Moreover, 

First, we give preliminaries of the proof.
Let 
% \begin{equation*}
    $F(\Bx) = F^1(\Bx)+F^2(\Bx)$ and $G(\By) = G^1(\By)+G^2(\By)$,
% \end{equation*}
% $F(\Bx) = F^1(\Bx)+F^2(\Bx)$ and $G(\By) = G^1(\By)+G^2(\By)$, respectively.
where $F^1(\Bx)=\sum_{i\in\CN} f_i^1(x_i)$, $F^2(\Bx)=\sum_{i\in\CN} f_i^2(x_i)$, $G^1(\By)=\sum_{l\in\QGactive} g_\subl^1(y_\subl)$, and $G^2(\By)=\sum_{l\in\QGactive} g_\subl^2(y_\subl)$.
Now, we consider the following FLiP-ADMM algorithm:\begin{subequations}\label{alg:Full-FLiP}
\begin{align}
\label{alg:Bx_update_FLiP}
\small
 \Bx^{k+1} &\in \underset{\Bx}{\operatorname{argmin}}
 \{F^1(\Bx)+\langle \nabla F^2(\Bx^k) + \BW^\top\BA^\top\Bu^k,\Bx\rangle \nonumber  \\
&+\frac{1}{2}\|\BA\BW\Bx+\BB\By^k-\Bc\|^2_\BGamma+\frac{1}{2}\left\|\Bx-\Bx^k\right\|_\BP^2 \} \\
  \label{alg:By_update}
\By^{k+1}& \in \underset{\By}{\operatorname{argmin}}\{G^1(\By)
+\langle \nabla G^2(\By^k) + \BB^\top \Bu^k, \By\rangle \nonumber \\
&\!\!\!\!\!+\frac{1}{2}\|\BA\BW\Bx^{k+1}+\BB\By-\Bc\|^2_\BGamma+\frac{1}{2}\|\By-\By^k\|_\BQ^2\}  \\
\label{alg:Bu_update}
 \Bu^{k+1}&=\Bu^k+\BPhi \BGamma (\BA\BW\Bx^{k+1}+\BB\By^{k+1}-\Bc).
\end{align}
\end{subequations}
Now, suppose $\BP$ and $\BQ$ are given as
$\BP = \BD_\alpha^{-1}-\BW^\top\BA^\top\BGamma\BA\BW$ and $\BQ=\bdiag([Q_\subl]_{l\in\QGactive})=\BD_\beta^{-1}-\BB^\top\BGamma\BB$, respectively.
% This is called the \textit{linearization technique}.
Then, \eqref{alg:Full-FLiP} is equivalent to Algorithm \ref{alg:FLiP_linearized}. 
Additionally, when $F^2=G^2=0$, \eqref{alg:Full-FLiP} is reduced to Algorithm \ref{alg:a-w-Proposed}.
Note that under Assumption \ref{assu:convergence}a--b, the $x_i$- and $y_\subl$- subproblems in \eqref{alg:xi-update}, \eqref{alg:xi_update_FLiP}, \eqref{alg:yl-update}, and \eqref{alg:yl-update_FLiP} are always well-defined (see Proposition 12.15 in \cite{Bauschke2011-pu}).
% Let $\Bx^*,\,\By^*,\,\Bu^*$ be a saddle point of the Lagrangian $\CL$.
Moreover, let $\Bw^k:=[\Bx^{k\top},\By^{k\top},\Bu^{k\top}]^\top$ and $\Bw^*:=[\Bx^{*\top},\By^{*\top},\Bu^{*\top}]^\top$.
Furthermore, we define 
\begin{align}
&M_0 = \frac{1}{2}\bdiag(\BP,\,\BB^\top\BGamma\BB + \BQ ,\,\BGamma^{-1}\BPhi^{-1}),\nonumber\\
&M_1 = \frac{1}{2}\bdiag(0,\,\BL_{G^2}+\BQ,\, \BTheta(\BPhi^{-1})^2\BGamma^{-1}), \nonumber
% \end{align}
% \begin{align}
\end{align}
\begin{align}
\label{def:M2}
&M_2  = \frac{1}{2}\bdiag (\BP-\BL_{F^2}, \nonumber\\
&\quad\quad\quad\quad  \BB^\top \BGamma(I-\BTheta^{-1}(I-\BPhi)^2)\BB +\BQ - 3\BL_{G^2}, \nonumber\\
& \quad\quad\quad\quad\quad\quad\quad (\BPhi^{-1})^2\BGamma^{-1}(2I-\BPhi-\Theta ) ),
% \begin{bmatrix}
% \BP-\BL_{F^2} & 0 & 0 \\
% 0 & 
% \substack{\bigl[\BGamma \left(I-\BTheta^{-1}(I-\BPhi)^2\right)\\ +\BQ - 3\BL_{G^2}\bigr]} & 0 \\
% 0 & 0 & \substack{\bigl[(\BPhi^{-1})^2\BGamma^{-1} \\ \times(2I-\BPhi-\Theta)\bigr]}
% \end{bmatrix},
% &M_0 =\frac{1}{2}
% \begin{bmatrix}
% \BP & 0 & 0 \\
% 0 & \BGamma + \BQ & 0 \\
% 0 & 0 & \BGamma^{-1}\BPhi^{-1}
% \end{bmatrix},\nonumber \\
% &M_1=\frac{1}{2}
% \begin{bmatrix}
% 0 & 0 & 0 \\
% 0 & \BL_{G^2}+\BQ & 0 \\
% 0 & 0 & \BTheta(\BPhi^{-1})^2\BGamma^{-1}
% \end{bmatrix},\nonumber \\
% \label{def:M2}
% &M_2  = \frac{1}{2}
% \begin{bmatrix}
% \BP-\BL_{F^2} & 0 & 0 \\
% 0 & 
% \substack{\bigl[\BGamma \left(I-\BTheta^{-1}(I-\BPhi)^2\right)\\ +\BQ - 3\BL_{G^2}\bigr]} & 0 \\
% 0 & 0 & \substack{\bigl[(\BPhi^{-1})^2\BGamma^{-1} \\ \times(2I-\BPhi-\Theta)\bigr]}
% \end{bmatrix},
\end{align}
where $\BL_{F^2}=\bdiag([L_{f^2_i}I_d ]_{i\in\CN})$, $\BL_{G^2}=\bdiag([L_{g^2_\subl}I_{q_\subl} ]_{l\in\QGactive})$, and $\BTheta=\bdiag([\theta_\subl I_{p_\subl} ])$ with $\theta_\subl = 2-\phi_\subl-\varepsilon_\subl >0,\,l\in\QGactive$.
% $M_2\succeq 0$ and $M_2\neq 0$ are satisfied for $\theta_i = 2-\phi_i-\varepsilon_i>0,\,i\in\CN$.
Finally, we define the Lyapunov function $V^k$ as
\begin{equation}\label{Lyapunov}
    V^k = \|\Bw^k-\Bw^*\|^2_{M_0} + \|\Bw^k-\Bw^{k-1}\|^2_{M_1}.
\end{equation}

With this in mind, Theorem \ref{thm:convergence} is proved as follows.
% with several supporting lemmas presented in Appendix \ref{subsec:supporting_lemma} .
By using the equality $\|\Bw^{k+1}-\Bw^*\|_{M_0}^2=\|\Bw^k-\Bw^*\|_{M_0}^2 - \|\Bw^{k+1}-\Bw^k\|_{M_0}^2 + 2 \langle \Bw^{k+1}-\Bw^k,\Bw^{k+1}-\Bw^* \rangle_{M_0}$, we have
\begin{align}\label{ineq:V-1}
% \small
&V^{k+1} = V^k - \|\Bw^k-\Bw^{k-1}\|^2_{M_1} + \|\Bw^{k+1}-\Bw^k\|_{M_1}^2 \\
 &- \|\Bw^{k+1}-\Bw^k\|_{M_0}^2 + 2 \langle \Bw^{k+1}-\Bw^k,\Bw^{k+1}-\Bw^* \rangle_{M_0} \nonumber
\end{align}
for $V^k$ in \eqref{Lyapunov}.
Then, adding \eqref{ineq:87} in Lemma \ref{lem:87} and \eqref{ineq:88} in Lemma \ref{lem:88} to \eqref{ineq:V-1}, we obtain 
\begin{align*}
\small
&V^{k+1} \leq V^k - \|\Bw^k-\Bw^{k-1}\|^2_{M_1} + \|\Bw^{k+1}-\Bw^k\|_{M_1}^2 \\
     &- \|\Bw^{k+1}-\Bw^k\|_{M_0}^2 + 2 \langle \Bw^{k+1}-\Bw^k,\Bw^{k+1}-\Bw^* \rangle_{M_0}  \\
     &+\textstyle\sum_{i\in\CN} \frac{L_{f^2_i}}{2}\|x_i^{k+1}-x_i^k\|^2 
+ \textstyle\sum_{l\in\QGactive} \frac{L_{g^2_\subl}}{2}\|y_\subl^{k+1}-y_\subl^k\|^2 \\
    &-2\langle \Bw^{k+1}-\Bw^k, \Bw^{k+1}-\Bw^* \rangle_{M_0}\\
    &+ \|\Bu^{k+1}-\Bu^k\|^2_{(I-\BPhi^{-1})\BPhi^{-1}\BGamma^{-1}}\\
    &+\frac{1}{2}\|\Bykp-\By^k\|^2_{\BL_{G^2} -\BQ + \BB^\top\BTheta^{-1}(I-\BPhi)^2 \BGamma\BB}\\
    &+\frac{1}{2}\|\By^k-\By^{k-1}\|^2_{\BL_{G^2}+\BQ} 
    + \frac{1}{2}\|\Bu^k-\Bu^{k-1}\|^2_{\BTheta\BGamma^{-1}(\BPhi^{-1})^2}\\
    &- \left( \CL(\Bxkp,\Bykp,\Bu^*) - \CL(\Bx^*,\By^*,\Bu^*) \right)\\
    &= V^k -\|\Bw^{k+1}-\Bw^k\|_{M_2}^2\\
    &\quad\quad- \left( \CL(\Bxkp,\Bykp,\Bu^*) - \CL(\Bx^*,\By^*,\Bu^*) \right),
\end{align*}
% that is, $V^{k+1} \leq V^k-\|\Bw^{k+1}-\Bw^k\|_{M_2}^2- \left( \CL(\Bxkp,\Bykp,\Bu^*) - \CL(\Bx^*,\By^*,\Bu^*) \right)$
for any $\BTheta\succ0$.
Consequently, from this inequality, Lemma \ref{lem:PQ}, Lemma \ref{lem:M2}, and
$\CL(\Bxkp,\Bykp,\Bu^*) - \CL(\Bx^*,\By^*,\Bu^*)\geq0$,
we obtain $\lim_{k\to\infty}\|\Bw^{k+1}-\Bw^k\|_{M_2}^2=0$ and $\lim_{k\to \infty} (\CL(\Bxkp,\Bykp,\Bu^*) - \CL(\Bx^*,\By^*,\Bu^*))  = 0$.
Then, By $\|\Bw^{k+1}-\Bw^k\|_{M_2}^2\to 0$, we have $\lim_{k\to\infty}(\Bu^{k+1}-\Bu^k)= 0$
Therefore, $\lim_{k\to\infty}(F(\Bx^k)+G(\By^k))=(F(\Bx^*)+G(\By^*))$ and $\lim_{k\to\infty}(\BA\BW\Bx^k+\BB\By^k-\Bc)=0$ are achieved.
\end{proof}

Moreover, when both $\BA$ and $\BB$ are the identity matrix, 
we can prove the convergence with fully agent-wise algorithmic parameters and convergence conditions.
In this case, the agents can independently choose their algorithmic parameters without any cooperation.
\begin{theorem}\label{thm:convergence_localize}
Assume that $\BA=\BB=I$.
Consider that in Algorithms \ref{alg:a-w-Proposed} and \ref{alg:FLiP_linearized},
we replace $\gamma_\subl$, $\phi_\subl$, and $\beta_\subl$ by the diagonal matrices $\diag(\gamma_\CCl) \otimes I_d$ with $\gamma=[\gamma_1,\ldots,\gamma_n]^\top\in\BR^n$, $\diag(\phi_\CCl) \otimes I_d$ with $\phi=[\phi_1,\ldots,\phi_n]^\top\in\BR^n$, and $\diag(\beta_\CCl) \otimes I_d$ with $\beta=[\beta_1,\ldots,\beta_n]^\top\in\BR^n$, respectively.
Additionally, in \eqref{alg:yl-update} and \eqref{alg:yl-update_FLiP},
we replace $\prox_{\beta_\subl g_\subl}(\cdot)$ with $\prox_{g_\subl}^{(\diag(\beta_\CCl)\otimes I_d)^{-1}}(\cdot)$.
% (Note that $\prox_{\beta_\subl g_\subl}$ in \eqref{alg:yl-update} is now replaced by $\prox_{g_\subl}^{(\diag(\beta_\CCl)\otimes I_d)^{-1}}$.)
Assume that Assumption \ref{assu:convergence} is satisfied.
% Assume that all the subproblems of $x_i,\,i\in\CN$ and $y_\subl,\,l\in\QGactive$ always have solutions.
% Assume that for all $i\in\CN$, $f_i$ is $L_{f_i}$-smooth, and for all $l\in\QGactive$, $g_\subl$ is $L_{g_\subl}$-smooth, where $L_{f_i}\geq0$ and $L_{g_\subl}\geq0$.
Assume that for all $i\in\CN$ and some $\varepsilon_i \in(0,2-\phi_i)$, the following inequalities are satisfied:
\begin{align*}
    &\alpha_i^{-1} \geq \gamma_i |\QiGactive| + L_{f_i},\quad \beta_i^{-1}-\gamma_i \geq 0, \quad \nonumber \\
    &\gamma_i\left(1-\frac{(1-\phi_i)^2}{2-\phi_i-\varepsilon_i}\right)I_d +Q_i \succeq 3
    \left(\max_{l\in\QiGactive}{L_{g^2_\subl}}\right)I_d,
\end{align*}
where $Q_i = (\beta_i^{-1} -\gamma_i)I_d$.
Then, $\lim_{k\to\infty}(F(\Bx^k)+G(\By^k))=(F(\Bx^*)+G(\By^*))$ and $\lim_{k\to\infty}(\BW\Bx^k-\By^k)=0$ hold.
\end{theorem}
\begin{proof}
    We can prove Theorem \ref{thm:convergence_localize} similarly to Theorem \ref{thm:convergence} by modifying Lemma \ref{lem:88} in the proof of Theorem \ref{thm:convergence}.
    Due to the space limit, we omit the proof.
\end{proof}

\section{Application to Consensus Optimization}\label{sec:case_study}
\subsection{Problem Setting}\label{subsec:consensus}

Now, we consider the following problem:
\begin{align}\label{prob:consensus}
\begin{array}{ll}
\underset{x_1\,\ldots,x_n\in\BR^d}{\mbox{minimize}}&
\displaystyle \
\sum_{i\in\CN} ( f_i(x_i) + h_i(x_i))
\\
\displaystyle\mbox{subject to}& x_i=x_j \quad \forall \{i,j\} \in \CE,
\end{array}
\end{align}
where $f_i$ and $h_i$ for $i\in\CN$ are proper, closed, and convex.
The problem in \eqref{prob:consensus} can be reformulated in the form of Problem \eqref{prob} with $A_\subl=-B_\subl=I_{|\CCl| d}$, $c_\subl=0$, and
\begin{equation}\label{eq:consensus_proximable}
    g_\subl(y_\subl) =  \sum_{j\in\CCl} \frac{1}{|\QjGactive|} h_j([y_\subl]_{\pi_\subl(j)})
+ \delta_{\CD_l}(y_\subl),
\end{equation}
% \begin{align*}
% % \setlength{\baselineskip}{30pt}
% \begin{array}{ll}
% \underset{x_1\,\ldots,x_n\in\BR^d}{\mbox{minimize}}&
% \displaystyle \
% \sum_{i\in\CN}  f_i(x_i) \\
% & \displaystyle+ \sum_{l\in\QGactive}
% \bigl( \sum_{j\in\CCl} \frac{1}{|\QjGactive|} g_j([y_\subl]_{\pi_\subl(j)})
% + \delta_{\CD_l}(y_\subl) \bigr)
% \\
% \displaystyle\mbox{subject to}& x_\CCl-y_\subl=0 \quad l\in\QGactive,
% \end{array}
% \end{align*}
where $\CD_l := \{y_\subl \in \BR^{|\CCl|d}
% \in\BR^{|\CCl|d\times |\CCl|d}
: \exists \xi\in\BR^d \text{ s.t. } y_\subl = \mathbf{1}_{|\CCl|}\otimes\xi   \}$ for $l\in\QGactive$, and $\delta_{\CD_l}(y_\subl)$ is the indicator function for $\CD_l$, i.e., $\delta_{\CD_l}$ satisfies $\delta_{\CD_l}(y_\subl) = 0$ for $y_\subl\in\CD_l$ and $\delta_{\CD_l}(y_\subl) = \infty$ for $y_\subl\notin\CD_l$.
% \begin{equation*}
%     \delta_{\CD_l}(y_\subl) = 
%     \begin{cases}
%     0, & x_\CCl\in\CD_l \\
%     \infty & x_\CCl\notin \CD_l 
%     \end{cases}
% \end{equation*}
% with $\CD_l := \{y_\subl
% % \in\BR^{|\CCl|d\times |\CCl|d}
% : \exists \xi\in\BR^d \text{ s.t. } y_\subl = \mathbf{1}_{|\CCl|}\otimes\xi   \}$ for $l\in\QGactive$.
% In this formulation, $y_\subl$ is a copy of $x_\CCl$ for each $l\in\CCl$.

For the $g_\subl$ in \eqref{eq:consensus_proximable} and $\CD_l$, the following proposition states that the proximal mapping of $g_\subl$ in \eqref{eq:consensus_proximable} is proximable if that of $(\sum_{j\in\CCl} \frac{1}{|\QjGactive|} h_j)(\cdot)$ is proximable.
\begin{prop}\label{lem:consensus_proximable}
    Assume that the functions $r_j:\BR^d\to\BR,\,j\in\CCl$ are proper, closed, and convex functions.
    Let $\bar{r}_\subl:\BR^d\to\BR\cup\{+ \infty\}$ and $s_\subl:\BR^{|\CCl|d}\to\BR\cup\{+ \infty\}$
    be $\bar{r}_\subl(z)=\sum_{j\in\CCl} r_j(z)$ for $z\in\BR^d$ and 
    % $s_\subl (x_\CCl) =\beta_\subl \delta_{\CD_l}(x_\CCl)+\beta_\subl \sum_{j\in\CCl} r_j(x_j)$
    \begin{equation*}
        s_\subl (x_\CCl) =\sum_{j\in\CCl} r_j(x_j) + \delta_{\CD_l}(x_\CCl),
    \end{equation*}
    respectively. Then, 
    \begin{align*}
    % \label{eq:lem_prox_consensus}
      % \!\!\!\!\!\!\!\!\!  
      \prox_{s_\subl} (x_\CCl) = \mathbf{1}_{|\CCl|}\otimes \prox_{\frac{1}{|\CCl|}\bar{r}_\subl} \left( \frac{1}{|\CCl|} (\mathbf{1}_{|\CCl|}\otimes I_d)^\top x_\CCl \right)
    \end{align*}
    holds for $\bar{r}_\subl$ and any $x_{\CCl}\in\BR^{|\CCl|d}$.

\end{prop}
\begin{proof}
    From the definition of $\CD_l$, we have $\prox_{s_\subl}(x_\CCl) = \argmin_{v\in\BR^{|\CCl|d}} \{ \frac{1}{2}\|x_\CCl-z\|^2 + s_\subl(v)\} = \argmin_{\xi\in\BR^d}\{ \sum_{j\in\CCl}  \frac{1}{2}\|x_j - \xi \|^2 + \bar{r}_\subl(\xi) \}$. Then, we obtain
    $\xi^* \in \argmin_{\xi\in\BR^d}\{ \sum_{j\in\CCl}  \frac{1}{2}\|x_j - \xi \|^2 + \bar{r}_\subl(\xi) \}  \Leftrightarrow 0 \in |\CCl|\{ \xi^* - \frac{1}{|\CCl|} \sum_{j\in\CCl}x_j  + \frac{1}{|\CCl|}  \pd \bar{r}_\subl(\xi^*) \}\Leftrightarrow \frac{1}{|\CCl|} \sum_{j\in\CCl} x_j \in (\mathbb{I} +\frac{1}{|\CCl|}  \pd \bar{r}_\subl) (\xi^*) \Leftrightarrow \xi^* \in (\mathbb{I} +\frac{1}{|\CCl|}  \pd \bar{r}_\subl)^{-1} ( \frac{1}{|\CCl|} \sum_{j\in\CCl} x_j )$,
%     \begin{align*}
%         &\xi^* \in \argmin_{\xi\in\BR^d}\{ \sum_{j\in\CCl}  \frac{1}{2}\|x_j - \xi \| + \beta_\subl \bar{r}_\subl(\xi) \} \\
% \Leftrightarrow& 0 \in |\CCl|\{ \xi^* - \frac{1}{|\CCl|} \sum_{j\in\CCl}x_j  + \frac{\beta_l}{|\CCl|}  \pd \bar{r}_\subl(\xi^*) \} \\
% \Leftrightarrow& \frac{1}{|\CCl|} \sum_{j\in\CCl} x_j \in (\mathbb{I} +\frac{\beta_l}{|\CCl|}  \pd \bar{r}_\subl(\xi^*)) (\xi^*) \\
% \Leftrightarrow& \xi^* \in (\mathbb{I} +\frac{\beta_l}{|\CCl|}  \pd \bar{r}_\subl(\xi^*))^{-1} ( \frac{1}{|\CCl|} \sum_{j\in\CCl} x_j ),
%     \end{align*}
    where $\mathbb{I}$ is the identity operator.
    Therefore, Proposition \ref{lem:consensus_proximable} follows from Proposition 16.44 in \cite{Bauschke2011-pu}.
\end{proof}

\subsection{Numerical Experiment}\label{subsubsec:experiment_consensus}
Through a numerical experiment,
% of solving the well-known $\ell_1$-regularized least squares problem,
we demonstrate the effectiveness of the proposed CL-ADMM and CL-FLiP-ADMM.

% \begin{figure}
%     \centering
%     \includegraphics[width=0.5\columnwidth]{network_0224.png}
%     \caption{The communication network $G$ in the numerical experiment in Subsection \ref{subsec:consensus}.}
%     \label{fig:network_consensus}
% \end{figure}
Consider a multi-agent system with $n=50$ agents.
Assume that the communication network $\CG$ is given as a connected time-invariant undirected graph, where each edge is generated with a probability of $0.1$.
We consider that the multi-agent system solves the consensus optimization problem in \eqref{prob:consensus} with 
\begin{align}\label{eq:consensus_objective}
    f_i(x_i) = \frac{1}{2}\|\Psi_i x_i - b_i\|^2,\quad h_i(x_i) = \lambda_i \|x_i\|_1,
\end{align}
% $f_i(x_i) = \frac{1}{2}\|\Psi_i x_i - b_i\|^2$ and $g_i(x_i) = \lambda_i \|x_i\|_1$
where $\Psi_i = I_d + 0.1 \Omega_i \in \BR^{d\times d}$, $b_i\in\BR^d,\,i\in\CN$, and $\lambda_i=\lambda=0.001$ for all $i\in\CN$.
For all $i\in\CN$, each entry of $\Omega_i$ and $b_i$ is generated by the standard normal distribution.
Note that $f_i$ in \eqref{eq:consensus_objective} is $\lambda_\mathrm{max}(\Psi_i^\top \Psi_i)$-smooth, and the functions $g_\subl$ in \eqref{eq:consensus_proximable} is proximable for $h_i$ in \eqref{eq:consensus_objective}.
% This problem is the $\ell_1$-regularized least squares problem and has been utilized in many fields, such as parameter estimation \cite{Tibshirani1996-ah}.

Now, to verify the effectiveness of the proposed clique-based algorithms, we conduct simulations for 
the CL-ADMM with $\QGactive=\maxQG$, CL-ADMM with $\QGactive = \CE$, CL-FLiP-ADMM with $\QGactive=\maxQG$, CL-FLiP-ADMM with $\QGactive=\CE$, and PG-EXTRA \cite{Shi2015-bm}, given as follows: \begin{align*}
        &\Bx^{k+1} = \prox_{\eta \lambda\|\cdot\|_1}
        ( (W_\mathrm{m}\otimes I_d) \Bx^k - \eta  \nabla F(\Bx) - \Bv^k)\\
        &\Bv^{k+1} = \Bv^k + \frac{I_{nd} - W_\mathrm{m}\otimes I_d}{2} \Bx^k,
    \end{align*} 
    where $W_\mathrm{m}\in\BR^{n\times n}$ is a mixing matrix of $\CG$.
% the following algorithms:\begin{itemize}
%     \item[(a)] CL-ADMM with $\QGactive=\maxQG$;
%     \item[(b)] CL-ADMM with $\QGactive = \CE$;
%     \item[(c)] CL-FLiP-ADMM with $\QGactive=\maxQG$, $f_i^1(x_i)=0$, $f_i^2(x_i)=f_i(x_i)$, $g_\subl^1(y_\subl)=g_\subl(y_\subl)$, and $g_\subl^2(y_\subl)=0$;
%     \item[(d)] CL-FLiP-ADMM with $\QGactive=\CE$, $f_i^1(x_i)=0$, $f_i^2(x_i)=f_i(x_i)$, $g_\subl^1(y_\subl)=g_\subl(y_\subl)$, and $g_\subl^2(y_\subl)=0$;
%     \item[(e)] PG-EXTRA \cite{Shi2015-bm}, given as follows: \begin{align*}
%         &\Bx^{k+1} = \prox_{\eta \lambda\|\cdot\|_1}
%         ( (W_\mathrm{m}\otimes I_d) \Bx^k - \eta  \nabla F(\Bx) - \Bv^k)\\
%         &\Bv^{k+1} = \Bv^k + \frac{I_{nd} - W_\mathrm{m}\otimes I_d}{2} \Bx^k,
%     \end{align*} 
%     where $W_\mathrm{m}\in\BR^{n\times n}$ is a mixing matrix of $\CG$.
%     % \item[(f)] Normal ADMM with $\QGactive=\maxQG$;
%     % % in (a) with $\alpha_i=\hat{\alpha}_{\mathrm{f}},\,i\in\CN$.
%     % \item[(g)] Normal FLiP-ADMM  with $\QGactive=\maxQG$.
%     % % in (c) with$\alpha_i=\hat{\alpha}_\mathrm{g},\,i\in\CN$.
% \end{itemize}
Note that when $\CG$ is connected, $\cap_{l\in\QGactive} \{\Bx\in\BRnd: x_\CCl \in\CD_l \} = \{\Bx\in\BRnd: x_1=\cdots=x_n\}$ is satisfied for $\QGactive=\maxQG$ from Proposition 4.2 in \cite{Sakurama2021-oy}, and hence an optimal solution can be obtained by the CL-ADMM and CL-FLiP-ADMM with $\QGactive=\maxQG$.
% Note that the normal ADMM in (f) and normal FLiP-ADMM in (g) correspond to special cases of the algorithms in (a) and (c) in which all the parameters $\gamma_\subl$, $\phi_\subl$, $\beta_\subl$, and $\alpha_i$ are common among all $l\in\QGactive$ and all $i\in\CN$, respectively.

The algorithmic parameters are given as follows.
For the CL-ADMM algorithms with $\QGactive=\maxQG$ and $\CE$, we set $\gamma_\subl = \phi_\subl=\beta_\subl=1$ for all $l\in\QGactive$, and $\alpha_i = 1/|\QiGactive|$ for each $i\in\CN$.
For the CL-FLiP-ADMM algorithms with $\QGactive=\maxQG$ and $\CE$, we set $f_i^1(x_i)=0$, $f_i^2(x_i)=f_i(x_i)$, $g_\subl^1(y_\subl)=g_\subl(y_\subl)$, $g_\subl^2(y_\subl)=0$, $\gamma_\subl = \phi_\subl=\beta_\subl=  1$ for all $l\in\QGactive$, and $\alpha_i = 1/ (\lambda_\mathrm{max}(\Psi_i^\top \Psi_i) + |\QiGactive|)$ for each $i\in\CN$.
% Note that by choosing the different $\alpha_i$ for each agent $i$, the performance is expected to be less conservative than the normal ADMM and FLiP-ADMM algorithms in \cite{Ryu2022-oc}.
For the PG-EXTRA, we set $\eta = 0.9(1 + \lambda_\mathrm{min}(W_\mathrm{m}))/ \lambda_\mathrm{max}(\mathbf{\Psi}^\top \mathbf{\Psi})$ with $\mathbf{\Psi}= \bdiag(\Psi_1,\ldots,\Psi_n)$
and 
$W_\mathrm{m} = I_n -  \frac{1}{\max_{i\in\CN} |\CN_i|-1 } L_\CG $, where $L_\CG$ is the graph laplacian matrix of the graph $\CG$.
% For the method in (g), we set $\gamma_\subl = \phi_\subl=\beta_\subl=1$ for all $l\in\QGactive$, and $\alpha_i=1/\max_{i\in\mathcal{N}}|\QiGactive|$ for all $i\in\CN$.
% Finally, for the method in (f), we set $\gamma_\subl = \phi_\subl=\beta_\subl=  1$ for all $l\in\QGactive$, and $\alpha_i=1/ \max_{i\in\mathcal{N}} (\lambda_\mathrm{max}(\Psi_i^\top \Psi_i ) + |\QiGactive|) $ for all $i\in\CN$.

\begin{figure}[t]
    \centering
    \includegraphics[width=\columnwidth]{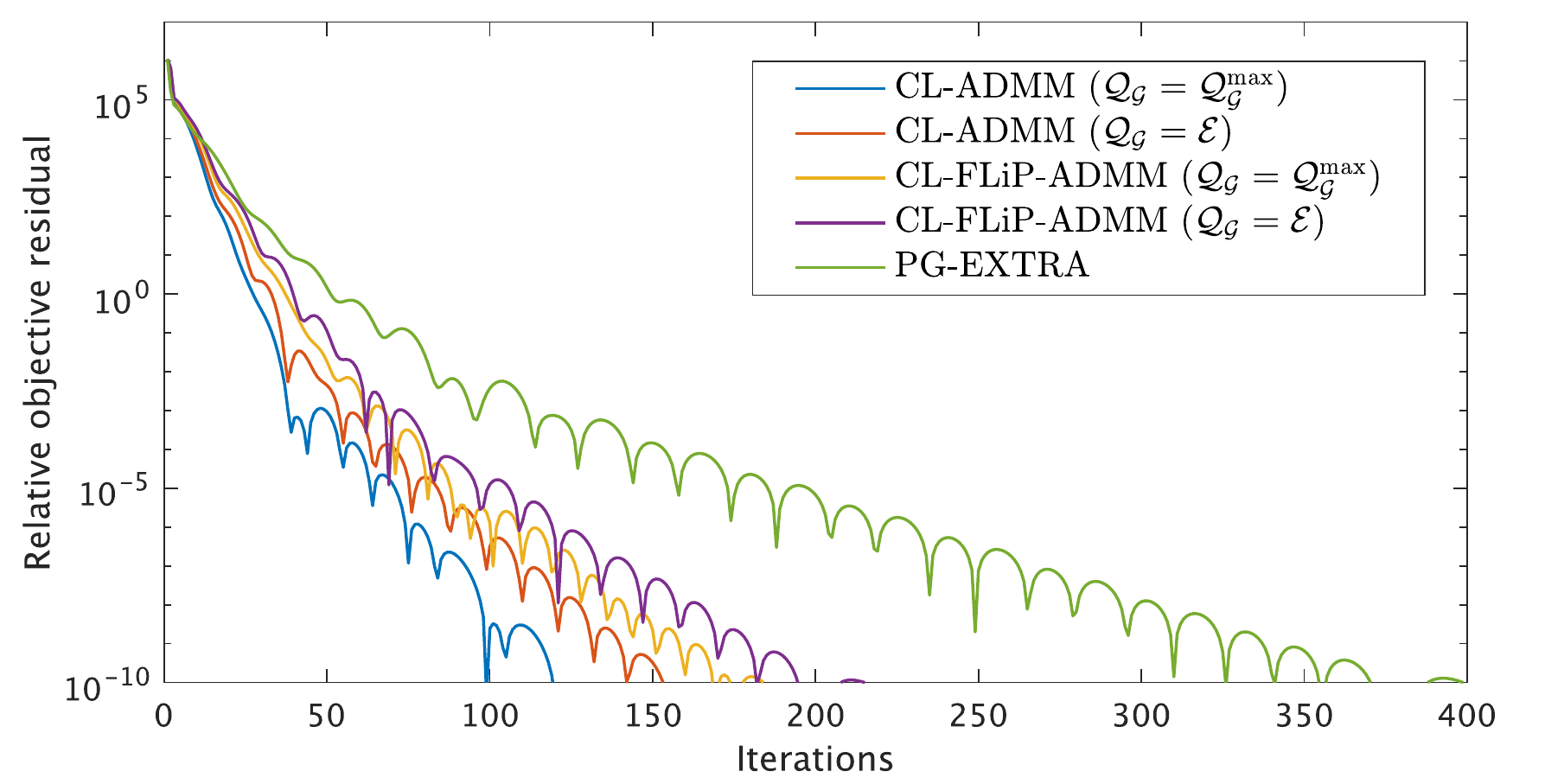}
    \caption{Plot of relative objective residuals $|(F(\Bx^k)+G(\Bx^k) ) -(F(\Bx^*)-G(\Bx^*))|/(F(\Bx^*)+G(\Bx^*))$ against the number of iteration under the CL-ADMM and CL-FLiP-ADMM algorithms with $\QGactive=\maxQG$ and $\CE$, and PG-EXTRA.}
    \label{fig:sim_result_consensus}
\end{figure}
\begin{figure}[t]
    \centering
    \hspace*{-0.2cm}\includegraphics[width=0.98\columnwidth]{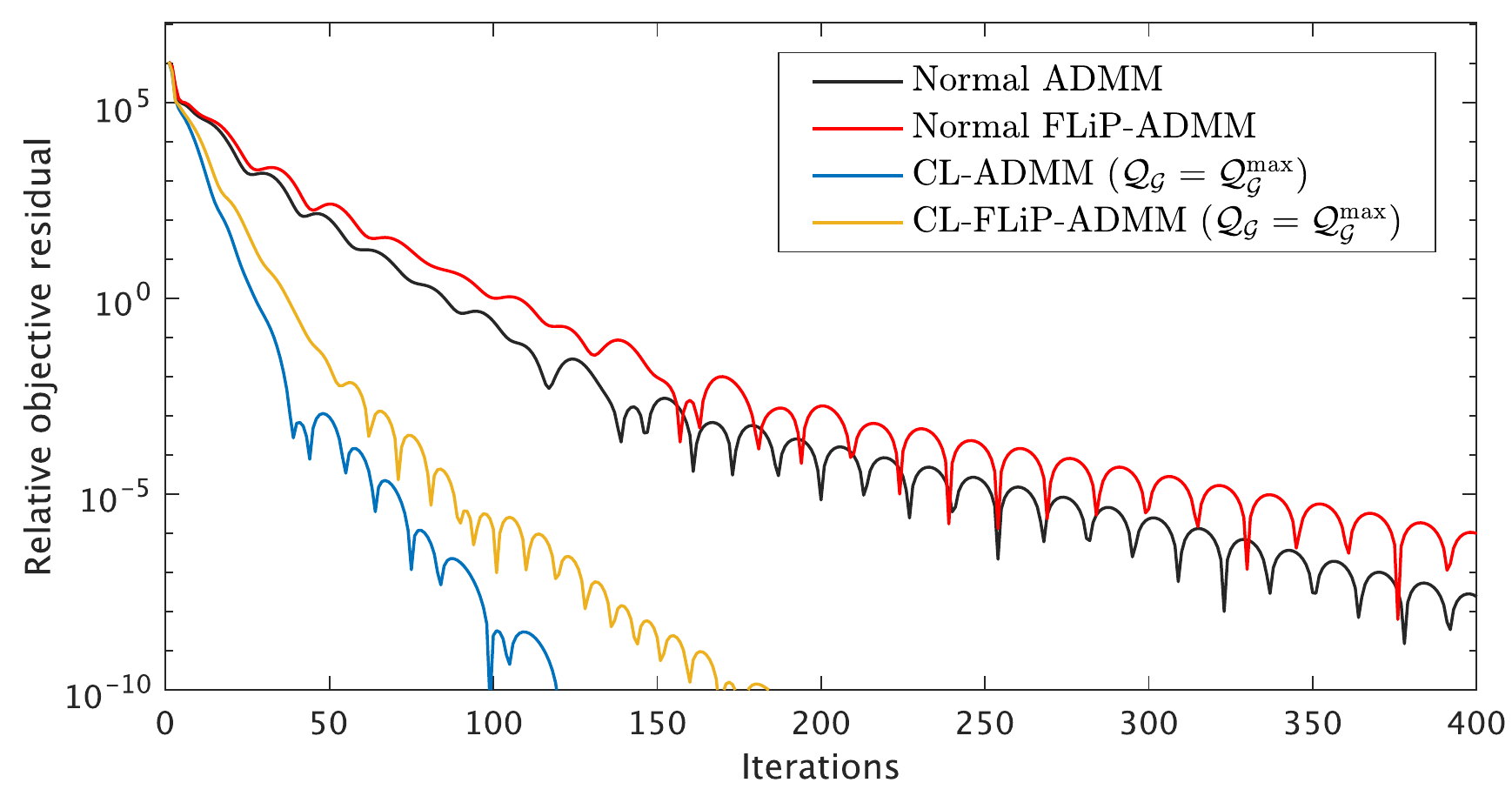}
    \caption{Plot of relative objective residuals $|(F(\Bx^k)+G(\Bx^k) ) -(F(\Bx^*)+G(\Bx^*))|/(F(\Bx^*)+G(\Bx^*))$ against the number of iteration under the normal ADMM, normal FLiP-ADMM, and the CL-ADMM and CL-FLiP-ADMM algorithms with $\QGactive=\maxQG$.}
    \label{fig:sim_result_consensus2}
\end{figure}

Fig. \ref{fig:sim_result_consensus} plots the relative objective residuals $|(F(\Bx^k)+G(\Bx^k) ) -(F(\Bx^*)+G(\Bx^*))|/(F(\Bx^*)+G(\Bx^*))$ versus iterations.
As shown in Fig. \ref{fig:sim_result_consensus}, all the methods successfully converge to an optimal solution with a tiny error, and the CL-ADMM with $\QGactive=\maxQG$ outpaces the others.
Additionally, although the CL-FLiP-ADMM with $\QGactive=\maxQG$ is slower than the CL-ADMM algorithms, it outperforms the CL-FLiP-ADMM with $\QGactive=\CE$ and PG-EXTRA.
% Besides, the (normal) ADMM methods are faster than the FLiP-ADMM methods, although the FLiP-ADMM allows us to utilize the gradient of $f_i^2$.
Moreover, Fig. \ref{fig:sim_result_consensus} indicates that handling the consensus constraint on a clique basis, i.e., setting not $\QGactive=\CE$ but $\QGactive=\maxQG$ can enhance the performance.

For further comparison, we also run the normal ADMM and FLiP-ADMM algorithms on the same problem, and Fig. \ref{fig:sim_result_consensus2} plots the relative objective residual versus iterations of the normal ADMM and normal FLiP-ADMM with the results of the CL-ADMM and CL-FLiP-ADMM with $\QGactive=\maxQG$ in Fig. \ref{fig:sim_result_consensus}.
Here, the normal ADMM and FLiP-ADMM correspond to special cases of Algorithms \ref{alg:a-w-Proposed} and \ref{alg:FLiP_linearized} in which all the parameters $\gamma_\subl$, $\phi_\subl$, $\beta_\subl$, and $\alpha_i$ are common among all $l\in\QGactive$ and all $i\in\CN$.
For the normal ADMM, we set $\QGactive=\maxQG$, $\gamma_\subl = \phi_\subl=\beta_\subl=1$ for all $l\in\QGactive$, and $\alpha_i=1/\max_{i\in\mathcal{N}}|\QiGactive|$ for all $i\in\CN$.
For the normal FLiP-ADMM, we set $\QGactive=\maxQG$, $\gamma_\subl = \phi_\subl=\beta_\subl=  1$ for all $l\in\QGactive$, and $\alpha_i=1/ \max_{i\in\mathcal{N}} (\lambda_\mathrm{max}(\Psi_i^\top \Psi_i ) + |\QiGactive|) $ for all $i\in\CN$.
From Fig. \ref{fig:sim_result_consensus2}, the normal ADMM and normal FLiP-ADMM are much slower than the CL-ADMM and CL-FLiP-ADMM.
This implies that thanks to the localized algorithmic parameters, the proposed methods are not only fully distributed but perform better than the normal ADMM and FLiP-ADMM.

These results highlight the effectiveness of the proposed methods.
% and the framework clique-wise coupling.
\begin{rem}
The clique-wise handling of pairwise constraints (e.g., the consensus constraint in \eqref{prob:consensus}) tends to outperform the pairwise one, in particular, when $G$ is not dense, and the initial value of $\Bx^k$ is far from its optimal value.
When the clique-wise handling does not outperform the edge-based one, the clique-wise coupled framework is still meaningful because 
the agents can share a part of their computation (see Remark \ref{rem:share}).
% of the $y_\subl$ update in \eqref{alg:yl-update} with others more flexibly.
\end{rem}
% \section{Proof of Proposition \ref{lem:WTu}}\label{sec:proof_Wu}
%     From \eqref{def:W}, the $(j,i)$ block of $W_\subl \in\BR^{|\CCl|d\times nd}$ can be written as $[W_\subl]_{ji}= w_{\subl,ji}I_d $, where
% \begin{equation}\label{def:w_ij}
% \small
%     w_{\subl,ji} = 
%     \begin{cases}
%     1,& i\in\CCl \text{ and }\pi_\subl(i)=j\\
%     0,& \text{otherwise}
%     \end{cases}.
% \end{equation}
% % for $\CCl=\{m_{l,1},\ldots,m_{l,|\CCl|}\},\,1\leq m_{l,1}<\ldots<m_{l,|\CCl|}\leq n$.
% Additionally, we have $\BW^\top\Bu = \sum_{l\in\QGactive} (W_\subl)^\top u_\subl$.
% Then, for the $i$th block $[\BW^\top\Bv]_i\in\BR^d$ of $\BW^\top\Bv$, we obtain
% \begin{align*}
% \small
% \textstyle
%      [\BW^\top\Bv]_i 
%     = \sum_{l\in\QGactive}\left( \sum_{j=1}^{|\CCl|} [W_\subl]_{ji} [v_\subl]_j \right) 
%     = \sum_{l\in\QiGactive} [v_\subl]_{\pi_\subl(i)}
% \end{align*}
% because $[W^\subl]_{ji} [v^\subl]_{j}= [v^\subl]_{\pi_\subl(i)}$ holds for $j$ satisfying $\pi_\subl(i)=j$, and $[W^\subl]_{ji} [v^\subl]_{j}= 0$ holds otherwise.

\section{Conclusion}\label{sec:conclusion}
    This paper addressed a novel framework for distributed optimization, clique-wise coupled optimization problems.
    We proposed a distributed ADMM and FLiP-ADMM algorithms based on cliques
    and proved the convergence theorems with no global parameter.
    % Our proposed method and its convergence theorem does not contain any global algorithmic parameters, which allow each agent to choose its parameters independently.
    Moreover, we applied the proposed methods to a consensus optimization problem and demonstrated their effectiveness via numerical experiments.
    A future direction is to extend existing methods for pairwise coupled distributed optimization/control problems to a more general framework from the viewpoint of clique-wise coupling.
% Our future directions are two-fold;
% 1) to extend existing methods of pairwise coupled distributed optimization/control problems to a more general framework from the view point of clique-wise coupling;
% 2) to apply the proposed method to non-convex problems.

\appendix

% \section{Proof of Theorem\ref{thm:convergence}}\label{sec:proof_convergence}
\subsection{Supporting Lemmas}\label{subsec:supporting_lemma}
In the following, several supporting lemmas for Theorem \ref{thm:convergence} are provided. Here, let $\hat{\Bu}^{k+1}$ be 
\begin{equation}\label{def:uhat}
    \hat{\Bu}^{k+1}=\Bu^k + \BGamma(\BA\BW \Bxkp + \BB\Bykp - \Bc).
\end{equation}
% \end{align*}
% Before proceeding to prove Theorem \ref{thm:convergence}, we prove several lemmas.
\begin{lemma}\label{lem:PQ}
Assume that the inequalities in \eqref{stepsize_condition} are satisfied.
Then, $\BP\succeq0$ and $\BQ\succeq 0$ hold for 
$\BP = \BD_\alpha^{-1}-\BW^\top\BA^\top\BGamma\BA\BW $ and $\BQ=\BD_\beta^{-1}-\BB^\top\BGamma\BB$.
\end{lemma}
\begin{proof}
    $\BQ\succeq 0$ directly follows from \eqref{stepsize_condition2}.
    In the following, we prove $\BP\succeq 0$.
    Let $\hat{D}_\subl \in \BR^{|\CCl|d\times |\CCl|d}$ be
    $\hat{D}_\subl =  \gamma_\subl \lambda_\mathrm{max}(A_\subl^\top A_\subl) I_{|\CCl|d}$.
    Then, we obtain
    \begin{align*}
    \small
       \BP =& \BD_\alpha^{-1} - \textstyle\sum_{l\in\QGactive} W_\subl^\top \hat{D}_\subl W_\subl + \sum_{l\in\QGactive} W_\subl^\top \hat{D}_\subl W_\subl \\
            &- \textstyle\sum_{l\in\QGactive} \gamma_\subl W_\subl^\top A_\subl^\top A_\subl W_\subl \succeq \BD_\alpha^{-1} - \sum_{l\in\QGactive} W_\subl^\top \hat{D}_\subl W_\subl. 
    \end{align*}
    Now, for the $(i,j)$ block $[W_\subl^\top \hat{D}_\subl W_\subl]_{ij}\in\BR^{d\times d}$ of $W_\subl^\top \hat{D}_\subl W_\subl$,
    \begin{equation*}
    \small
        [W_\subl^\top \hat{D}_\subl W_\subl]_{ij} = 
        \begin{cases}
            \gamma_\subl \lambda_\mathrm{max}(A_\subl^\top A_\subl) I_d   & i=j \text{ and } i,j\in\CCl \\
            0    & \text{otherwise}
        \end{cases}
    \end{equation*}
    is obtained from \eqref{def:w_ij}.
    Therefore, $[\sum_{l\in\QGactive} W_\subl^\top \hat{D}_\subl W_\subl]_{ii} = \sum_{l\in\QiGactive} \gamma_\subl \lambda_\mathrm{max}(A_\subl^\top A_\subl)I_d  $ and $[\sum_{l\in\QGactive} W_\subl^\top \hat{D}_\subl W_\subl]_{ij} = 0$ for $i\neq j$
    hold.
    Hence, we obtain $\BP\succeq \BD_\alpha^{-1} - \sum_{l\in\QGactive} W_\subl^\top \hat{D}_\subl W_\subl\succeq 0$ from \eqref{stepsize_condition1}.
\end{proof}
\begin{lemma}\label{lem:M2}
If there exists $\varepsilon_\subl \in(0,2-\phi_\subl)$ for all $l\in\QGactive$ such that the inequalities in \eqref{stepsize_condition} is satisfied, then $M_2\succeq0$ and $M_2\neq0$ are satisfied for $M_2$ in \eqref{def:M2}, $\BP = \BD_\alpha^{-1}-\BW^\top\BA^\top\BGamma\BA\BW $, and $\BQ=\BD_\beta^{-1}-\BB^\top\BGamma\BB$.
\end{lemma}
\begin{proof}
For the $(1,1)$ block of $M_2$, $\BP-\BL_{F^2}=\bdiag([P_i-L_{f_i}I_d]_{i\in\CN})\succeq0$ follows from \eqref{stepsize_condition}.
For the $(3,3)$ block of $M_2$, 
$(\BPhi^{-1})^2\BGamma^{-1}[
2I-\BPhi-\BTheta]\succ0$ holds from $\theta_\subl=2-\phi_\subl-\varepsilon_\subl$ and $\varepsilon_\subl\in(0,2-\phi_\subl),\,i\in\CN$.
Finally, consider the $(2,2)$ block of $M_2$, i.e., $\BB^\top \left(\BGamma \left(I-\BTheta^{-1}(I-\BPhi)^2\right)\right)\BB +\BQ - 3\BL_{G^2}$.
For its $(l,l)$ block, we obtain
$ \left[ \BB^\top \left(\BGamma \left(I-\BTheta^{-1}(I-\BPhi)^2\right)\right)\BB +\BQ - 3\BL_{G^2}\right]_{ll}
    =  \gamma_\subl \left(1 - \textstyle\frac{(1-\phi_\subl)^2}{\theta_\subl}\right)B_\subl^\top B_\subl + Q_\subl -3 L_{g^2_\subl} I_{q_\subl} \succeq 0 $
% \begin{align*}
%     &\left[ \BB^\top \left(\BGamma \left(I-\BTheta^{-1}(I-\BPhi)^2\right)\right)\BB +\BQ - 3\BL_{G^2}\right]_{ll}\\
%     =&  \gamma_\subl \left(1 - \textstyle\frac{(1-\phi_\subl)^2}{\theta_\subl}\right)B_\subl^\top B_\subl + Q_\subl - L_{g^2_\subl} I_{q_\subl} \succeq 0
% % =& \left(\diag(\gamma_\CCl)\otimes I_d \right)
% %     \bigl(I_{|\CCl|d} \\
% %     &-\left(\diag(\theta_\CCl)\otimes I_d\right)^{-1}
% %     \left(I_{|\CCl|d}-\left(\diag(\phi_\CCl)\otimes I_d\right)
% %     \right)^2 \bigr)\\
% %     &+\bdiag([Q_j]_{j\in\CCl}) - L_{g_\subl} I_{|\CCl|d}\\
% % =&\bdiag\left( \left[\gamma_j\left(1 - \frac{(1-\phi_j)^2}{\theta_j}\right)I_d + Q_j - L_{g_\subl}I_d\right]_{j\in\CCl}\right)\\
% \end{align*}
from \eqref{stepsize_condition}.
Therefore, $M_2\succeq 0$ and $M_2\neq0$ are satisfied.
\end{proof}
\begin{lemma}
Consider the algorithm in \eqref{alg:Full-FLiP}.
Assume that Assumption \ref{assu:convergence} is satisfied.
% Assume that all the subproblems of $x_i,\,i\in\CN$ and $y_\subl,\,l\in\QGactive$ always have solutions.
Then, for any $\Bx\in\BRnd$,
% and $\hat{\Bu}^{k+1}=\Bu^k + \BGamma(\BW \Bxkp -\Bykp)$,
the following inequality holds:
\begin{align}\label{ineq:85}
\small
&F(\Bx^{k+1})-F(\Bx)
+ \langle \hat{\Bu}^{k+1},\BA\BW(\Bx^{k+1}-\Bx) \rangle \nonumber \\
\leq & \textstyle\sum_{i\in\CN} \frac{L_{f^2_i}}{2} \|x_i^{k+1}-x_i^k\|^2
- \langle \Bx^{k+1}-\Bx^k, \Bx^{k+1}-\Bx \rangle_\BP
\nonumber\\
&+ \langle \BGamma\BB(\By^{k+1}-\By^k),\BA\BW(\Bx^{k+1}-\Bx) \rangle.
\end{align}
Moreover, for any $\By\in\BRnd$,
the following inequality holds:
\begin{align}\label{ineq:86}
\small
\!\!\!\!& G(\By^{k+1})-G(\By) 
+ \langle \hat{\Bu}^{k+1}, \BB( \By^{k+1}-\By )\rangle \\
\!\!\!\!&\leq \textstyle\sum_{l\in\QGactive} \frac{L_{g^2_\subl}}{2} \|y_\subl^{k+1}-y_\subl^k\|^2 - \langle \By^{k+1}-\By^k, \By^{k+1}-\By \rangle_\BQ. \nonumber
\end{align}
\end{lemma}
\begin{proof}
Consider the $\Bx$-subproblem in \eqref{alg:Bx_update}.
For any $\Bx$,
\begin{align}\label{ineq:stage1_1}
\small
&0\leq F^1(\Bx)-F^1(\Bx^{k+1})
+\langle \nabla F^2(\Bx) + \BW^\top \BA^\top [ \Bu^k+
\nonumber\\
&\BGamma( \BA \BW\Bx^{k+1}+\BB\By^k-\Bc)] 
% \nonumber\\&
+ \BP(\Bx^{k+1} -\Bx^k),
\Bx-\Bx^{k+1} \rangle
\end{align}
holds because for $\Bxkp$ in \eqref{alg:Bx_update}, it holds that $\Bxkp \in \argmin_\Bx \left\{
F^1(\Bx) + \langle \nabla J(\Bxkp), \Bx-\Bxkp \rangle
\right\}$ with $J(\Bx) = \left\langle\nabla F^2\left(\Bx^k\right)+\BW^{\top} \Bu^k, \Bx\right\rangle 
 +\frac{1}{2}\left\|\BW\Bx-\By^k\right\|^2_\BGamma+\frac{1}{2}\left\|\Bx-\Bx^k\right\|_\BP^2 $, which can be verified from the optimality condition.
 Next, by the convexity and $L_{f^2_i}$-smoothness of $f^2_i$,
 we obtain
 \begin{align}\label{ineq:state1_2}
    &\langle \nabla F^2(\Bx^k), \Bx-\Bxkp \rangle \nonumber\\
    =& \textstyle\sum_{i\in\CN} \left( \langle \nabla f^2_i(x_i^k),x_i-x_i^k \rangle 
    +\langle \nabla f^2_i(x_i^k),x_i^k-x_i^{k+1} \rangle \right)
    \nonumber \\
    % \leq&  \sum_{i\in\CN} ( f_i(x_i)
    % -f_i(x_i^k) \nonumber \\
    % &+f_i(x_i^k) -f_i(x_i^{k+1}) 
    % + \frac{L_{f_i}}{2} \|x_i^{k+1}-x_i^k \|^2) \nonumber \\
    =& \textstyle F^2(\Bx) -F^2(\Bxkp) +  \sum_{i\in\CN} \frac{L_{f^2_i}}{2} \|x_i^{k+1}-x_i^k \|^2.
 \end{align}
Then, adding \eqref{ineq:stage1_1} to \eqref{ineq:state1_2}, we obtain $\textstyle 0 \leq F(\Bx) - F(\Bxkp) +  \sum_{i\in\CN} \frac{L_{f_i}}{2} \|x_i^{k+1}-x_i^k \|^2
+ \langle \BW^\top \BA^\top \left[ \Bu^k+\BGamma(\BA\BW\Bx^{k+1}+\BB\By^k-\Bc)\right] 
+ \BP(\Bx^{k+1} -\Bx^k),\Bx-\Bxkp \rangle$.
% \begin{align*}
% &\textstyle 0 \leq F(\Bx) - F(\Bxkp) +  \sum_{i\in\CN} \frac{L_{f_i}}{2} \|x_i^{k+1}-x_i^k \|^2 \nonumber \\
% +& \langle \BW^\top \BA^\top \left[ \Bu^k+\BGamma(\BA\BW\Bx^{k+1}+\BB\By^k-\Bc)\right] \nonumber \\
% +& \BP(\Bx^{k+1} -\Bx^k),\Bx-\Bxkp \rangle.
% \end{align*}
Hence, \eqref{ineq:85} follows from this inequality.
The second inequality in \eqref{ineq:86} can be obtained in the same procedure as the proof of \eqref{ineq:85}.
\end{proof}
\begin{lemma}\label{lem:87}
Consider the algorithm in \eqref{alg:Full-FLiP}.
Assume that Assumption \ref{assu:convergence} is satisfied.
% Assume that all the subproblems of $x_i,\,i\in\CN$ and $y_\subl,\,l\in\QGactive$ always have solutions.
Then, for $\CL$ in \eqref{lagrangian} and $\hat{\Bu}^k$ in \eqref{def:uhat}, the following inequality holds:
\begin{align}\label{ineq:87}
\small
    &\CL(\Bxkp,\Bykp,\Bu^*) - \CL(\Bx^*,\By^*,\Bu^*) \nonumber \\
\leq& \textstyle\sum_{i\in\CN} \frac{L_{f^2_i}}{2}\|x_i^{k+1}-x_i^k\|^2 
+ \sum_{l\in\QGactive} \frac{L_{g^2_\subl}}{2}\|y_\subl^{k+1}-y_\subl^k\|^2
\nonumber  \\
&-2\langle \Bw^{k+1}-\Bw^k, \Bw^{k+1}-\Bw^* \rangle_{M_0} \nonumber\\
&-\langle \BPhi^{-1}(\Bu^{k+1}-\Bu^k), \BB(\By^{k+1}-\By^k) \rangle \nonumber\\
&+ \|\Bu^{k+1}-\Bu^k\|^2_{(I-\BPhi^{-1})\BPhi^{-1}\BGamma^{-1}}
\end{align}
\end{lemma}
\begin{proof}
From \eqref{alg:Bu_update} and \eqref{def:uhat}, we have
\begin{align*}
    &\langle \Bu^* -\hat{\Bu}^{k+1}, \BA\BW \Bxkp + \BB\Bykp -\Bc \rangle \\
    &\quad\quad\quad 
    =\langle \Bu^*-\hat{\Bu}^{k+1} , \BGamma^{-1}(\hat{\Bu}^{k+1}-\Bu^k) \rangle,\\
    & \BA\BW(\Bx^{k+1}-\Bx^*) = \BPhi^{-1}\BGamma^{-1}(\Bukp-\Bu^k) -\BB (\Bykp-\By^*),\\
    & \Bu^*-\hat{\Bu}^{k+1} = (I-\BPhi^{-1})(\Bu^{k+1}-\Bu^k)-(\Bu^{k+1}-\Bu^*),\\
    &\hat{\Bu}^{k+1}-\Bu^k = \BPhi^{-1}(\Bukp-\Bu^k).
\end{align*}
% These are derived from \eqref{alg:Bu_update} and \eqref{def:uhat}.
Then, substituting these equations into the sum of \eqref{ineq:85} with $\Bx=\Bx^*$ and \eqref{ineq:86} with $\By=\By^*$, we obtain \eqref{ineq:87}.
\end{proof}
% Next
\begin{lemma}\label{lem:88}
Consider the algorithm in \eqref{alg:Full-FLiP}.
Assume that Assumption \ref{assu:convergence} is satisfied.
% all the subproblems of $x_i,\,i\in\CN$ and $y_\subl,\,l\in\QGactive$ always have solutions.
Then, it holds that
\begin{align}\label{ineq:88}
\small
    & \langle \BPhi^{-1}(\Bukp-\Bu^k),\BB(\Bykp-\By^k) \rangle \nonumber \\
&\leq\textstyle\frac{1}{2}\|\Bykp-\By^k\|^2_{\BL_{G^2} -\BQ + \BB^\top \BTheta^{-1}(I-\BPhi)^2 \BGamma \BB}  \\
&+\textstyle\frac{1}{2}\|\By^k-\By^{k-1}\|^2_{\BL_{G^2}+\BQ} 
+ \frac{1}{2}\|\Bu^k-\Bu^{k-1}\|^2_{\BTheta\BGamma^{-1}(\BPhi^{-1})^2}.\nonumber
\end{align}
% where
% % \begin{equation*}
%     $\BTheta = \bdiag([\theta_\subl I_{q_\subl}]_{l\in\QGactive})$
%     % (\diag(\theta_{\CC_1})\otimes I_{d},\ldots,\diag(\theta_{\CC_{\qstar}})\otimes I_{d})
% % \end{equation*}
% for any $\theta_i>0,\,i\in\CN$.
\end{lemma}
\begin{proof}
For \eqref{ineq:87}, by replacing $(k+1,k)$ with $(k,k-1)$ and setting $\By=\Bykp$, we obtain $G(\By^k)-G(\Bykp) + \langle \hat{\Bu}^k,\BB(\By^k-\Bykp) \rangle
\leq \textstyle\sum_{l\in\QGactive}\frac{L_{g^2_\subl}}{2} \|y_\subl^k-y_{\subl,k-1}\|^2
- \|\By^{k+1}-\By^{k}\|^2_\BQ.$
% \begin{align*}
%     &G(\By^k)-G(\Bykp) + \langle \hat{\Bu}^k,\BB(\By^k-\Bykp) \rangle \\
% \leq& \textstyle\sum_{l\in\QGactive}\frac{L_{g^2_\subl}}{2} \|y_\subl^k-y_{\subl,k-1}\|^2
% - \|\By^{k+1}-\By_{k}\|^2_\BQ.
% \end{align*}
Then, adding this to \eqref{ineq:86} with $\By=\By^k$ gives
\begin{align}\label{ineq:stage2-2}
\small& \langle \BPhi^{-1}(\Bukp-\Bu^k), \BB(\Bykp-\By^k) \rangle \nonumber \\
\leq&  \textstyle\sum_{l\in\QGactive}
( \frac{L_{g^2_\subl}}{2} \|y_\subl^{k+1}-y_\subl^k\|^2 +\frac{L_{g^2_\subl}}{2} \|y_\subl^k-y_{\subl,k-1}\|^2 ) \nonumber \\
-&\|\By^{k+1}-\By^{k}\|^2_\BQ
    + \langle \By^{k+1}-\By^k,\By^k-\By^{k-1} \rangle_\BQ \nonumber \\
-& \left\langle (I-\BPhi^{-1})(\Bu^k-\Bu^{k-1}), \BB(\Bykp-\By^k) \right\rangle.
\end{align}
Now, by applying \textit{Young's inequality}
% i.e., $\langle a,b \rangle \leq \textstyle\frac{\epsilon}{2} \|a\|^2 + \frac{1}{2\epsilon}\|b\|^2$ for any $ a,b\in\BR^m$ and any $\epsilon>0$,
\begin{equation*}
    \langle a,b \rangle \leq \frac{\epsilon}{2} \|a\|^2 + \frac{1}{2\epsilon}\|b\|^2,\quad \forall a,b\in\BR^m,\,\forall \epsilon>0
\end{equation*}
to $\langle \By^{k+1}-\By^k,\By^k-\By^{k-1} \rangle_\BQ=\langle \BQ^{\frac{1}{2}}(\By^{k+1}-\By^k),\BQ^{\frac{1}{2}}(\By^k-\By^{k-1}) \rangle$ with $\epsilon=1$ and $\langle (I-\BPhi^{-1})(\Bu^k-\Bu^{k-1}), \BB(\Bykp-\By^k) \rangle
=\textstyle\sum_{l\in\QGactive} 
\langle \frac{1}{\sqrt{\gamma_\subl} {\phi^{\subl}}}
(u_\subl^k-u_{\subl,k-1}), 
\sqrt{\gamma_\subl}(\phi_\subl-1)\left( B_\subl ( y_\subl^{k+1}-y_\subl^k)\right)\rangle$
% \begin{align}\label{inq:for_young}
% &\left\langle (I-\BPhi^{-1})(\Bu^k-\Bu^{k-1}), \BB(\Bykp-\By^k) \right\rangle \nonumber \\
% % =&\left\langle \sqrt{\BGamma}^{-1}\BPhi^{-1}(\Bu^k-\Bu^{k-1}), \sqrt{\BGamma}(\BPhi-I)\BB(\Bykp-\By^k) \right\rangle \nonumber \\
% % =
% &\textstyle\sum_{l\in\QGactive} 
% \biggl\langle \frac{1}{\sqrt{\gamma_\subl} {\phi^{\subl}}}
% (u_\subl^k-u_{\subl,k-1}), \nonumber\\ 
% & \quad\quad\quad\quad
% \sqrt{\gamma_\subl}(\phi_\subl-1)\left( B_\subl ( y_\subl^{k+1}-y_\subl^k)\right)\biggr\rangle
% \end{align}
with $\epsilon=\sqrt{\theta_\subl}>0$, the inequality in \eqref{ineq:stage2-2} is reduced to
% $\langle \BPhi^{-1}(\Bukp-\Bu^k),\BB(\Bykp-\By^k) \rangle 
% \leq  \textstyle\sum_{l\in\QGactive}
% ( \frac{L_{g^2_\subl}}{2} \|y_\subl^{k+1}-y_\subl^k\|^2 +\frac{L_{g^2_\subl}}{2} \|y_\subl^k-y_{\subl,k-1}\|^2 ) 
% -\|\By^{k+1}-\By^k\|^2_\BQ 
% + \frac{1}{2} \|\Bykp-\By^k\|_\BQ^2 + \frac{1}{2} \|\By^k-\By^{k-1}\|_\BQ 
% +\textstyle\frac{1}{2} \|(\BTheta\BGamma)^{\frac{1}{2}}\BPhi^{-1}(\Bu^k-\Bu^{k-1})\|^2 
% + \textstyle\frac{1}{2} \|(\BTheta^{-1}\BGamma)^{\frac{1}{2}}(I-\BPhi)\BB(\Bykp-\By^k)\|^2.
% % & \frac{1}{2} \|\Bykp-\By^k\|_{\BL_{G^2}-\BQ + \BTheta^{-1}(I-\BPhi)^2\BGamma} + \frac{1}{2} \|\By^k-\By^{k-1}\|_{\BL_{G^2}+\BQ}\\
% % &+ \frac{1}{2} \|\Bu^k-\Bu^{k-1}\|^2_{\BPhi \BGamma^{-1}(\BPhi^{-1})^2 }
% $
\begin{align*} %\label{ineq:stage2-3}
& \langle \BPhi^{-1}(\Bukp-\Bu^k),\BB(\Bykp-\By^k) \rangle \nonumber \\
\leq&  \textstyle\sum_{l\in\QGactive}
\left( \frac{L_{g^2_\subl}}{2} \|y_\subl^{k+1}-y_\subl^k\|^2 +\frac{L_{g^2_\subl}}{2} \|y_\subl^k-y_{\subl,k-1}\|^2 \right) \nonumber \\
&-\|\By^{k+1}-\By^k\|^2_\BQ 
+ \frac{1}{2} \|\Bykp-\By^k\|_\BQ^2 + \frac{1}{2} \|\By^k-\By^{k-1}\|_\BQ \nonumber \\
&+\textstyle\frac{1}{2} \|(\BTheta\BGamma)^{\frac{1}{2}}\BPhi^{-1}(\Bu^k-\Bu^{k-1})\|^2 \nonumber \\
&+ \textstyle\frac{1}{2} \|(\BTheta^{-1}\BGamma)^{\frac{1}{2}}(I-\BPhi)\BB(\Bykp-\By^k)\|^2.
% & \frac{1}{2} \|\Bykp-\By^k\|_{\BL_{G^2}-\BQ + \BTheta^{-1}(I-\BPhi)^2\BGamma} + \frac{1}{2} \|\By^k-\By^{k-1}\|_{\BL_{G^2}+\BQ}\\
% &+ \frac{1}{2} \|\Bu^k-\Bu^{k-1}\|^2_{\BPhi \BGamma^{-1}(\BPhi^{-1})^2 }
\end{align*}
Hence, reorganizing this inequality, we obtain \eqref{ineq:88}.
\end{proof}
\bibliographystyle{IEEEtran}
\bibliography{bibliography.bib}

\begin{thebibliography}{10}
\providecommand{\url}[1]{#1}
\csname url@samestyle\endcsname
\providecommand{\newblock}{\relax}
\providecommand{\bibinfo}[2]{#2}
\providecommand{\BIBentrySTDinterwordspacing}{\spaceskip=0pt\relax}
\providecommand{\BIBentryALTinterwordstretchfactor}{4}
\providecommand{\BIBentryALTinterwordspacing}{\spaceskip=\fontdimen2\font plus
\BIBentryALTinterwordstretchfactor\fontdimen3\font minus
  \fontdimen4\font\relax}
\providecommand{\BIBforeignlanguage}[2]{{%
\expandafter\ifx\csname l@#1\endcsname\relax
\typeout{** WARNING: IEEEtran.bst: No hyphenation pattern has been}%
\typeout{** loaded for the language `#1'. Using the pattern for}%
\typeout{** the default language instead.}%
\else
\language=\csname l@#1\endcsname
\fi
#2}}
\providecommand{\BIBdecl}{\relax}
\BIBdecl

\bibitem{Nedic2009-kd}
A.~Nedi{\'c} and A.~Ozdaglar, ``{Distributed subgradient methods for
  multi-agent optimization},'' \emph{IEEE Trans. Automat. Contr.}, vol.~54,
  no.~1, pp. 48--61, Jan. 2009.

\bibitem{Shi2015-bm}
W.~Shi, Q.~Ling, G.~Wu, and W.~Yin, ``{A proximal gradient algorithm for
  decentralized composite optimization},'' \emph{IEEE Trans. Signal Process.},
  vol.~63, no.~22, pp. 6013--6023, Nov. 2015.

\bibitem{qu2017harnessing}
G.~Qu and N.~Li, ``Harnessing smoothness to accelerate distributed
  optimization,'' \emph{IEEE Trans. Control Netw. Syst.}, vol.~5, no.~3, pp.
  1245--1260, Apr. 2017.

\bibitem{Li2019-lj}
Z.~Li, W.~Shi, and M.~Yan, ``{A decentralized proximal-gradient method with
  network independent step-sizes and separated convergence rates},'' \emph{IEEE
  Trans. Signal Process.}, vol.~67, no.~17, pp. 4494--4506, Jan. 2019.

\bibitem{Latafat2019-mb}
P.~Latafat, N.~M. Freris, and P.~Patrinos, ``{A new randomized block-coordinate
  primal-dual proximal algorithm for distributed optimization},'' \emph{IEEE
  Trans. Automat. Contr.}, vol.~64, no.~10, pp. 4050--4065, Oct. 2019.

\bibitem{Li2023_primaldual}
H.~Li, E.~Su, C.~Wang, J.~Liu, Z.~Zheng, Z.~Wang, and D.~Xia, ``A primal-dual
  forward-backward splitting algorithm for distributed convex optimization,''
  \emph{IEEE Trans. Emerg. Top. Comput. Intell.}, vol.~7, no.~1, pp. 278--284,
  Feb. 2023.

\bibitem{Oh2015-fx}
K.-K. Oh, M.-C. Park, and H.-S. Ahn, ``{A survey of multi-agent formation
  control},'' \emph{Automatica}, vol.~53, pp. 424--440, Mar. 2015.

\bibitem{Bullo2009-ka}
F.~Bullo, J.~Cort{\'e}s, and S.~Martinez, \emph{{Distributed Control of Robotic
  Networks}}.\hskip 1em plus 0.5em minus 0.4em\relax Princeton Univ. Press,
  2009.

\bibitem{Sakurama2021-oy}
K.~Sakurama and T.~Sugie, ``{Generalized coordination of multi-robot
  systems},'' \emph{Found. Trends\textregistered{} Syst. Control}, vol.~9,
  no.~1, pp. 1--170, 2021.

\bibitem{bollobas1998modern}
B.~Bollob{\'a}s, \emph{Modern Graph Theory}.\hskip 1em plus 0.5em minus
  0.4em\relax Springer Science \& Business Media, 1998.

\bibitem{Bertsekas2015-nx}
D.~P. Bertsekas and J.~Tsitsiklis, \emph{\BIBforeignlanguage{en}{{Parallel and
  Distributed Computation: Numerical Methods}}}.\hskip 1em plus 0.5em minus
  0.4em\relax Prentice-Hall Englewood Cliffs, NJ, 1989.

\bibitem{Xu2010-fh}
H.~Xu, C.~Caramanis, and S.~Sanghavi, ``Robust {PCA} via outlier pursuit,''
  \emph{IEEE Trans. Inf. Theory}, vol.~58, no.~5, pp. 3047--3064, Jan. 2012.

\bibitem{Candes2011-nf}
E.~J. Cand{\`e}s, X.~Li, Y.~Ma, and J.~Wright, ``{Robust principal component
  analysis?}'' \emph{J. ACM}, vol.~58, no.~3, pp. 1--37, Jun. 2011.

\bibitem{Rudin1992-du}
L.~I. Rudin, S.~Osher, and E.~Fatemi, ``{Nonlinear total variation based noise
  removal algorithms},'' \emph{Physica D}, vol.~60, no.~1, pp. 259--268, Nov.
  1992.

\bibitem{Wahlberg2012-uk}
B.~Wahlberg, S.~Boyd, M.~Annergren, and Y.~Wang, ``{An ADMM algorithm for a
  class of total variation regularized estimation problems},'' \emph{IFAC Proc.
  Vol.}, vol.~45, no.~16, pp. 83--88, Mar. 2012.

\bibitem{Hallac2015-fm}
D.~Hallac, J.~Leskovec, and S.~Boyd, ``\BIBforeignlanguage{en}{{Network Lasso:
  Clustering and optimization in large graphs}},'' in
  \emph{\BIBforeignlanguage{en}{Proc. 21th ACM SIGKDD Int. Conf. Knowl.
  Discovery Data Mining}}, Aug. 2015, pp. 387--396.

\bibitem{Zhang2017-pn}
Y.~Zhang and Q.~Yang, ``\BIBforeignlanguage{en}{{An overview of multi-task
  learning}},'' \emph{\BIBforeignlanguage{en}{Natl. Sci. Rev.}}, vol.~5, no.~1,
  pp. 30--43, Sep. 2017.

\bibitem{Watanabe2023-xy}
Y.~Watanabe and K.~Sakurama, ``{Accelerated distributed projected gradient
  descent for convex optimization with clique-wise coupled constraints},'' in
  \emph{{Proc. 22nd IFAC World Congr.}}, 2023, (Accepted).

\bibitem{Boyd2011-yu}
S.~Boyd, N.~Parikh, E.~Chu, B.~Peleato, and J.~Eckstein, ``{Distributed
  optimization and statistical learning via the alternating direction method of
  multipliers},'' \emph{Found. Trends\textregistered{} Mach. Learn.}, vol.~3,
  no.~1, pp. 1--122, 2011.

\bibitem{Ryu2022-oc}
E.~K. Ryu and W.~Yin, \emph{\BIBforeignlanguage{en}{{Large-Scale Convex
  Optimization: Algorithms \& Analyses via Monotone Operators}}}.\hskip 1em
  plus 0.5em minus 0.4em\relax Cambridge Univ. Press, 2022.

\bibitem{Chang2016-lz}
T.-H. Chang, ``{A proximal dual consensus ADMM method for multi-agent
  constrained optimization},'' \emph{IEEE Trans. Signal Process.}, vol.~64,
  no.~14, pp. 3719--3734, Jul. 2016.

\bibitem{Falsone2020-mr}
A.~Falsone, I.~Notarnicola, G.~Notarstefano, and M.~Prandini,
  ``\BIBforeignlanguage{en}{{Tracking-ADMM for distributed constraint-coupled
  optimization}},'' \emph{\BIBforeignlanguage{en}{Automatica}}, vol. 117, p.
  108962, Jul. 2020.

\bibitem{Notarnicola2020-bn}
I.~Notarnicola and G.~Notarstefano, ``{Constraint-coupled distributed
  optimization: A relaxation and duality approach},'' \emph{IEEE Trans. Control
  Netw. Syst.}, vol.~7, no.~1, pp. 483--492, Mar. 2020.

\bibitem{Wu2023-vn}
X.~Wu, H.~Wang, and J.~Lu, ``{Distributed optimization with coupling
  constraints},'' \emph{IEEE Trans. Automat. Contr.}, vol.~68, no.~3, pp.
  1847--1854, Mar. 2023.

\bibitem{Condat2013-ld}
L.~Condat, ``{A primal--dual splitting method for convex optimization involving
  lipschitzian, proximable and linear composite terms},'' \emph{J. Optim.
  Theory Appl.}, vol. 158, no.~2, pp. 460--479, Aug. 2013.

\bibitem{Vu2013-ip}
B.~C. V{\~u}, ``{A splitting algorithm for dual monotone inclusions involving
  cocoercive operators},'' \emph{Adv. Comput. Math.}, Apr. 2013.

\bibitem{Bauschke2011-pu}
H.~H. Bauschke and P.~L. Combettes, \emph{{Convex Analysis and Monotone
  Operator Theory in Hilbert Spaces}}.\hskip 1em plus 0.5em minus 0.4em\relax
  Springer International Publishing, 2011.

\end{thebibliography}
\end{document}